\documentclass[12pt]{amsart}

\usepackage[T1]{fontenc}
\usepackage[english]{babel}

\usepackage{amsfonts}
\usepackage{amsmath}
\usepackage{amssymb}
\usepackage{amstext}
\usepackage{amsthm}

\usepackage{geometry}
\usepackage{fullpage}

\usepackage[shortlabels]{enumitem}
\usepackage{graphicx}
\usepackage{hyperref}
\hypersetup{hidelinks}
\usepackage{seqsplit}
\usepackage{mathrsfs}
\usepackage{mathtools}

\usepackage{booktabs}
\usepackage{longtable}
\usepackage{array}
\usepackage{pdflscape}

\allowdisplaybreaks

\newtheorem{theorem}{Theorem}[section]
\newtheorem{lemma}[theorem]{Lemma}
\newtheorem{proposition}[theorem]{Proposition}
\newtheorem{corollary}[theorem]{Corollary}
\newtheorem{conjecture}[theorem]{Conjecture}
\theoremstyle{definition}

\theoremstyle{remark}
\newtheorem{remark}[theorem]{Remark}
\theoremstyle{plain}

\newcommand{\N}{\mathbb{N}}
\newcommand{\Z}{\mathbb{Z}}

\newcommand{\R}{\mathbb{R}}
\newcommand{\MS}{\operatorname{MS}}
\newcommand{\IMS}{\operatorname{IMS}}
\newcommand{\CT}{\operatorname{CT}}
\DeclareMathOperator{\vol}{vol}
\DeclareMathOperator{\lcmop}{lcm}
\newcommand{\ind}{\mathbf{1}}
\newcommand{\code}[1]{\texttt{#1}}
\newcommand{\hashcode}[1]{{\ttfamily\fontsize{7}{8}\selectfont\seqsplit{#1}}}

\title{The Ehrhart series of magic squares of orders seven and eight}

\author{Dun Qiu}
\thanks{Dun Qiu is supported in part by the Fundamental Research Funds for
	the Central Universities (105-63263093), the National Natural Science
	Foundation of China (12271023 and 12171034), and the Natural Science
	Foundation of Tianjin (24JCZDJC01390).}
\address{Center for Combinatorics, LPMC, Nankai University, Tianjin 300071, P. R. China}
\email{qiudun@nankai.edu.cn}

\author{Guoce Xin}
\thanks{Guoce Xin is partially supported by the National Natural Science
	Foundation of China (12571355).}
\address{School of Mathematical Sciences,
	Capital Normal University, Beijing, 100048, P.R.~China}
\email{guoce\_xin@163.com}

\author{Zihao Zhang}
\address{School of Mathematics and Statistics,
	Beijing Institute of Technology, Beijing, 102400, P.R.~China}
\email{zihao-zhang@foxmail.com}


\begin{document}

\begin{abstract}
Let $\IMS_n(m)$ count the $n\times n$ nonnegative integer matrices whose row
sums, column sums, main-diagonal sum, and antidiagonal sum are all $m$.  We
determine the Ehrhart series
$F_n(q)=\sum_{m\geq0}\IMS_n(m)q^m$ as reduced rational functions for $n=7$
and $n=8$.  Their numerator--denominator degrees are respectively
$(366,373)$ and $(540,548)$.  Both numerators have positive integer
coefficients, are palindromic and strictly unimodal.

The proofs share one finite architecture: a signed SimpCone decomposition is
evaluated by quotient characters over finite fields, a certified common
denominator and Ehrhart reciprocity reduce the rational identity to finitely
many coefficients, an explicit counting bound lifts modular congruences to
integer equalities, and exact gcd computations prove reducedness.  For order
eight, a face-index pole certificate gives a degree-$598$ common denominator
without enumerating the full face lattice, leaving $296$ independent
coefficients in degrees $0$ through $295$.  Once the candidate rational
function is known, the first six production primes certify this finite prefix
by the same bounded-coefficient argument used for order seven. 
\end{abstract}

\maketitle

\section{Introduction}
\label{sec:intro}

The enumeration of nonnegative integer solutions to linear Diophantine
systems is a classical meeting point of partition analysis, polyhedral
geometry, and computer algebra.  In 1915, MacMahon~\cite{MacMahon1915}
introduced partition analysis precisely for such problems and treated magic
squares of small order as a leading example.  Throughout this paper, a magic
square is a \emph{weak} magic square: the entries are nonnegative integers
and may repeat.  By contrast, a \emph{pure} magic square requires the
distinct entries $1,\ldots,n^2$.

For $n\geq 3$ and $m\geq 0$, let $\IMS_n(m)$ denote the number of $n\times n$
matrices with nonnegative integer entries such that every row, every column,
the main diagonal, and the antidiagonal sum to $m$, and let
\begin{equation}
  F_n(q)=\sum_{m\geq 0}\IMS_n(m)\,q^m.
  \label{eq:def-Fn}
\end{equation}
By Ehrhart theory, $\IMS_n$ is a quasipolynomial in $m$ and $F_n$ is a
rational function; we refer to the book of Beck and
Robins~\cite{BeckRobins2015} for this background.  Ahmed, De Loera and
Hemmecke~\cite{AhmedDeLoeraHemmecke2002} studied the polyhedral cones of
magic squares and cubes and computed Hilbert basis data for small orders.
Barvinok~\cite{Barvinok1994} gave a polynomial time algorithm, in fixed
dimension, for counting lattice points by signed unimodular decompositions;
LattE supplied the first implementation and made this approach available for
substantial polyhedral computations~\cite{DeLoeraHemmeckeTauzerYoshida2004}.
The first complete generating function for order six was obtained by
Xin~\cite{Xin2015} using the CTEuclid algorithm for MacMahon's partition
analysis; the reported computation used three modular runs of roughly $70$
CPU days each, and the order six coefficients are recorded as
OEIS~A216039~\cite{OEISA216039}.  To the best of our knowledge, no complete
reduced rational form for order seven has previously appeared.

Finite Fourier formulas for rational cones go back at least to Diaz and
Robins~\cite{DiazRobins1997}.  For the cones produced by SimpCone, Xin, Xu and
Zhang~\cite{XinXuZhang2025} expressed the generating function as an average
over the character group of the dual lattice quotient.  The number of terms
in this average is the dual index.  This observation is particularly useful
here: although the order-seven decomposition contains more than $166$ million
signed cones, their mean dual index is only $1.5517$.  Most cones can therefore
be evaluated directly, without a subsequent unimodular decomposition.

The remaining difficulty is specialization.  Some rays have zero line-sum
degree, and the associated character products have poles when the auxiliary
variables are set equal to one.  We use a generic one-parameter
specialization, retain the necessary Laurent coefficients, and take the
constant term after the complete signed sum.  The factors with nonzero
line-sum degree are handled by a one-pass recurrence in the series degree; the
zero-degree factors are handled by the generalized-Todd method of Xin, Zhang
and Zhang~\cite{XinZhangZhangGTodd}.  We call this evaluator the
Laurent-regularized quotient-character (LRQC) method.  It is an implementation
of the SimpCone character formula, rather than a different cone identity.
Related routes include CTEuclid~\cite{Xin2015}, DecDenu~\cite{XinZhangZhang2024}, and primal-dual Barvinok decomposition~\cite{TaoXinZhang2025}; their favorable parameter ranges are compared in
Section~\ref{sec:work-reduction}.

Let $\Phi_k$ denote the $k$th cyclotomic polynomial.  The main result is as
follows.

A polynomial $P(q)=\sum_{j=0}^{r}c_jq^j$ is \emph{palindromic} if $c_j=c_{r-j}$ for all $j$;
It is called \emph{strictly unimodal} if additionally $c_0<c_1<\cdots<c_{\lfloor r/2\rfloor}$.

\begin{theorem}[Computer-assisted result]
\label{thm:intro-main}
The reduced Ehrhart series of weak magic squares of order seven is
\begin{equation}
  F_7(q)=\frac{N_7(q)}{D_7(q)},
  \label{eq:intro-f7}
\end{equation}
with
\begin{equation}
  D_7(q)=(1-q)^{35}\,
  \Phi_2^{22}\Phi_3^{19}\Phi_4^{12}\Phi_5^{10}
  \Phi_6^{8}\Phi_7^{7}\Phi_8^{6}\Phi_9^{5}
  \Phi_{10}^{4}\Phi_{11}^{4}\Phi_{12}^{2}
  \Phi_{13}^{2}\Phi_{14}\Phi_{15},
  \label{eq:intro-denominator}
\end{equation}
where $N_7(q)=\sum_{j=0}^{366}a_jq^j$ has positive integer coefficients, is
palindromic and strictly unimodal, and is not real-rooted.  The coefficients
$a_0,\ldots,a_{183}$ are listed in Appendix~\ref{app:numerator}, and
$a_j=a_{366-j}$ determines the rest.
\end{theorem}

The denominator~\eqref{eq:intro-denominator} contains cyclotomic factors of
every order from $1$ through $15$, so the quasiperiod of $\IMS_7$ is exactly
\begin{equation}
  \lcmop(1,2,\ldots,15)=360360.
  \label{eq:intro-period}
\end{equation}

The order-eight counterpart is the second main result.

\begin{theorem}[Computer-assisted result]
\label{thm:intro-main-eight}
The reduced Ehrhart series of weak magic squares of order eight is
\begin{equation}
  F_8(q)=\frac{N_8(q)}{D_8(q)},
  \label{eq:intro-f8}
\end{equation}
where
\begin{equation}
\begin{aligned}
D_8(q)={}&(1-q)^{48}\Phi_2^{30}\Phi_3^{21}\Phi_4^{17}
 \Phi_5^{13}\Phi_6^{12}\Phi_7^9\Phi_8^8\Phi_9^7
\Phi_{10}^6\Phi_{11}^5\Phi_{12}^4\Phi_{13}^4
 \Phi_{14}^2\Phi_{15}^2\Phi_{16}\Phi_{17}.
\end{aligned}
\label{eq:intro-denominator-eight}
\end{equation}
The denominator has degree $548$.  The polynomial
$N_8(q)=\sum_{j=0}^{540}b_jq^j$ has positive integer coefficients, is
palindromic, and is strictly unimodal.
It is not real-rooted.  The complete numerator and denominator coefficient
vectors are supplied in the companion online data archive rather than printed
in the article.
\end{theorem}

Every cyclotomic factor from $\Phi_1$ through $\Phi_{17}$ occurs in
\eqref{eq:intro-denominator-eight}.  Hence the exact quasiperiod of $\IMS_8$
is
\begin{equation}
 \lcmop(1,2,\ldots,17)=12{,}252{,}240.
 \label{eq:intro-period-eight}
\end{equation}

The proof of the order-seven theorem proceeds in four finite steps.
First, exact vertex enumeration together with
Proposition~\ref{prop:vertex-common-denominator} shows that the reduced Ehrhart
denominator divides $U_{15}(q)=(1-q)^{35}\prod_{r=2}^{15}\Phi_r(q)^{35}$.
Second, the signed SimpCone identity and the quotient-character formula
compute the coefficients of $F_7$ in finite fields.  Third, reciprocity and
Corollary~\ref{cor:finite-identity} reduce equality over $\mathbb{Z}(q)$ to
the prefix through degree $1256$.  Fourth, four $61$-bit prime fields produce
a candidate numerator by Chinese remaindering; the first six of the eight
certificate primes then lift the independently verified prefix congruences
uniquely to equalities over $\mathbb{Z}$, using the coefficient bounds supplied
by Theorem~\ref{thm:count-bound}.  Exact polynomial gcds confirm that the
displayed fraction is reduced.  The immutable computational objects are
identified in Appendix~\ref{app:certificates}.

Three features make a computation of this size practical.
\begin{enumerate}
  \item A cone of dual index $\delta$ contributes exactly $\delta$ character
    products.  For the complete order-seven list,
    \[
      \sum_J \delta_J = 257{,}827{,}737,\qquad
      \frac{1}{166{,}156{,}916}\sum_J \delta_J = 1.551712340,
    \]
    so a typical cone requires only one or two character products, despite
    the maximum index $23$.
  \item For fixed cone and character data, the cost is affine in the
    requested truncation degree $T$.
  \item Subtracting the all-ones matrix from an interior magic square shows
    that $\MS_n$ is Gorenstein of codegree $n$; palindromy halves both the
    reciprocal reconstruction and the final identity test.
\end{enumerate}
The integral cone data $(\sigma_J,\delta_J,m_k,c_k,H_{W_J})$ are computed once
and reused in every finite field; only the modular arithmetic changes.

The order-eight proof, given in Section~\ref{sec:order-eight}, follows the same
four-step architecture.  A face-index pole certificate replaces the complete
vertex envelope, root-local principal parts replace uniform truncation for the
signed terms, and reciprocity leaves the $296$ coefficients in degrees $0$
through $295$.  The first six production fields then prove the finite identity
by the same counting-bound argument as for order seven; the archived
eight-field centered-CRT reconstruction is an independent discovery and
cross-check route.  Exact polynomial gcds again give the reduced form.  A
separate computational report records the complete finite evidence.

Section~\ref{sec:polytope} describes the magic-square polytope, reciprocity,
and the vertex denominator bound.  Sections~\ref{sec:simpcone} and
\ref{sec:quotient} give the signed cone decomposition and the
quotient-character evaluation, including a complete order-five example.
Section~\ref{sec:work-reduction} compares the available cone-evaluation
routes.  Reconstruction and integer recovery are treated in
Section~\ref{sec:reconstruction}; the exact result, independent checks, and
computational cost follow in Section~\ref{sec:result}.
Section~\ref{sec:order-eight} presents the order-eight theorem, a comparison
with order seven, the denominator table for orders $4$ through $8$, and a
conjectural description of the numerator.

\section{The polytope of magic squares}
\label{sec:polytope}

\subsection{Definition, dimension, and constant term}

For $n\geq 3$, define the polytope
\begin{equation}
\begin{aligned}
\MS_n=\bigl\{X=(x_{ij})\in\R_{\geq0}^{\,n\times n}:{}&
 \sum_{j=1}^n x_{ij}=1\quad(1\leq i\leq n),\\
 &\sum_{i=1}^n x_{ij}=1\quad(1\leq j\leq n),\\
 &\sum_{i=1}^n x_{ii}=1,
 \quad\sum_{i=1}^n x_{i,n+1-i}=1\bigr\}.
\end{aligned}
\label{eq:def-MS}
\end{equation}
Then $\IMS_n(m)=\#\bigl(m\MS_n\cap\Z^{n^2}\bigr)$ counts the lattice points
of the $m$th dilate.  The row and column equations have rank $2n-1$, and the
two diagonal equations are independent modulo their span.  Consequently
\begin{equation}
  d_n:=\dim \MS_n=n^2-(2n+1)=(n-1)^2-2.
  \label{eq:dimension}
\end{equation}
In particular $d_7=34$, so $F_7$ has a pole of order $35$ at $q=1$.

For comparison with partition analysis, let $\lambda_i$, $\mu_j$, $\nu_1$ and
$\nu_2$ mark the row, column, main diagonal, and antidiagonal equations.
The constant term representation in the sense of
Xin~\cite{Xin2015} is
\begin{equation}
\begin{aligned}
F_n(q)={}&\CT_{\boldsymbol\lambda,\boldsymbol\mu,\nu_1,\nu_2}\,
\frac{1}{1-q\,(\lambda_1\cdots\lambda_n\,
\mu_1\cdots\mu_n\,\nu_1\nu_2)^{-1}}\\
&\times\prod_{1\leq i,j\leq n}
 \frac{1}{1-\lambda_i\mu_j
 \nu_1^{\ind(i=j)}\nu_2^{\ind(i+j=n+1)}}.
\end{aligned}
\label{eq:constant-term}
\end{equation}
This representation is compact, but iterated partial fractions applied to
it directly become impractical as $n$ grows.  Our route first converts
\eqref{eq:constant-term} into a signed family of simplicial cones.

\subsection{Period bound}

An exact reverse search enumeration with \code{lrs}, the implementation by
Avis~\cite{Avis2000} of the algorithm of Avis and
Fukuda~\cite{AvisFukuda1996}, found $5{,}920{,}184$ vertices of $\MS_7$ and
no rays.  The enumeration terminated normally and was not inferred from a
partial sample.  Table~\ref{tab:vertex-denominators} gives the complete
denominator census; in particular, every individual vertex denominator is at
most $15$.  The least common multiple of the vertex coordinate denominators is
\begin{equation}
 P_7^*=360360=2^3\cdot 3^2\cdot 5\cdot 7\cdot 11\cdot 13.
 \label{eq:period-bound}
\end{equation}
Therefore the quasiperiod of $\IMS_7$
divides~$P_7^*$~\cite[Chapter~3]{BeckRobins2015}.

\begin{table}[t]
\centering
\caption{Complete vertex-denominator census for $\MS_7$.}
\label{tab:vertex-denominators}
\footnotesize
\begin{tabular}{rrrrrrrrr}
\toprule
$d_v$ & 1 & 2 & 3 & 4 & 5 & 6 & 7 & 8\\
\midrule
vertices & 656 & 421488 & 2452656 & 1625544 & 741216 & 384576 & 170400 & 70656\\
\midrule
$d_v$ & 9 & 10 & 11 & 12 & 13 & 14 & 15 & {}\\
\midrule
vertices & 33024 & 7200 & 8928 & 1056 & 2208 & 192 & 384 & {}\\
\bottomrule
\end{tabular}
\end{table}

\subsection{Codegree and reciprocity}

The following observation supplies considerably more structure.

\begin{proposition}
\label{prop:codegree}
For $n\geq3$, the polytope $\MS_n$ is Gorenstein of codegree $n$.
Consequently
\begin{equation}
  \IMS_n(-m)=(-1)^{d_n}\,\IMS_n(m-n)
  \qquad(m\in\Z),
  \label{eq:reciprocity-values}
\end{equation}
and
\begin{equation}
  F_n(q^{-1})=(-1)^{d_n+1}q^{\,n}F_n(q).
  \label{eq:reciprocity-series}
\end{equation}
For $n=7$, the identity~\eqref{eq:reciprocity-series} reads
$F_7(q^{-1})=-q^7F_7(q)$.
\end{proposition}

\begin{proof}
Let $J$ denote the $n\times n$ all ones matrix.  A lattice point of
$m\MS_n$ lies in the relative interior exactly when every entry is
positive, that is, at least $1$.  The map
\begin{equation}
  X\longmapsto X-J
  \label{eq:interior-bijection}
\end{equation}
preserves every equal line sum equation and decreases the common line sum by
$n$, and its inverse adds $J$.  Hence \eqref{eq:interior-bijection} is a
bijection from the interior lattice points of $m\MS_n$ to the lattice points
of $(m-n)\MS_n$.  By Ehrhart--Macdonald reciprocity, proved in the form we
use by Stanley~\cite{Stanley1974},
\begin{equation}
  \IMS_n(-m)=(-1)^{d_n}\,\#\bigl(\operatorname{int}(m\MS_n)\cap\Z^{n^2}\bigr)
            =(-1)^{d_n}\,\IMS_n(m-n),
\end{equation}
which is \eqref{eq:reciprocity-values}.  The generating function form
\eqref{eq:reciprocity-series} is the standard restatement of
\eqref{eq:reciprocity-values}; see Beck and
Robins~\cite[Chapter~4]{BeckRobins2015}.
\end{proof}

\begin{proposition}[Vertex-ray common denominator]
\label{prop:vertex-common-denominator}
Let $P$ be a rational polytope of dimension $d$.  For a vertex $v$, let
$d_v$ be the least positive integer with $d_vv$ integral.  If $d_v\leq h$
for every vertex, then the reduced Ehrhart denominator of $P$ divides
\begin{equation}
 U_h(q)=(1-q)^{d+1}\prod_{r=2}^{h}\Phi_r(q)^{d+1}.
 \label{eq:vertex-common-denominator}
\end{equation}
For $P=\MS_7$ this gives
\begin{equation}
 U_{15}(q)=(1-q)^{35}\prod_{r=2}^{15}\Phi_r(q)^{35},
 \qquad \deg U_{15}=35\sum_{r=1}^{15}\varphi(r)=2520.
 \label{eq:U15}
\end{equation}
\end{proposition}

\begin{proof}
Homogenize $P$ to the pointed cone
$C=\operatorname{cone}\{(v,1):v\in P\}$.  The primitive integral ray
  corresponding to $v$ is $\rho_v=(d_vv,d_v)$.  A pulling triangulation of $C$
  uses only its extreme rays.  A generic half-open convention partitions $C$
  into the resulting full-dimensional simplicial cones without adding new
  rays.  Equivalently, ordinary inclusion--exclusion adds only faces generated
  by subsets of those rays.  Each such generating function has denominator
  $\prod_{\rho_v}(1-z^{\rho_v})$
\cite{BarvinokPommersheim1999,Payne2008}.  Setting the spatial variables to
$1$ and the homogenizing variable to $q$ changes these factors to
$1-q^{d_v}$.  Since
$1-q^{d_v}=\prod_{r\mid d_v}\Phi_r(q)$ up to the harmless normalization of
$\Phi_1$, every simplicial term has at most $d+1$ copies of each
$\Phi_r$ with $1\leq r\leq h$.  Equation~\eqref{eq:vertex-common-denominator}
is therefore a common denominator for the sum, and the reduced denominator
divides it.  Finally $\sum_{r=1}^{15}\varphi(r)=72$, which proves
\eqref{eq:U15}.
\end{proof}

\begin{corollary}[Finite identity criterion]
\label{cor:finite-identity}
Let $G_7(q)=N_7(q)/D_7(q)$ be the rational function displayed in
Theorem~\ref{thm:intro-main}.  If
\begin{equation}
 [q^m]F_7(q)=[q^m]G_7(q)\qquad(0\leq m\leq1256),
 \label{eq:finite-identity-prefix}
\end{equation}
then $F_7(q)=G_7(q)$.
\end{corollary}

\begin{proof}
 By Proposition~\ref{prop:vertex-common-denominator}, the reduced
denominator of $F_7$ divides $U_{15}$; the displayed $D_7$ also divides
$U_{15}$.  Hence
$\widetilde F=F_7U_{15}$ and $\widetilde G=G_7U_{15}$ are polynomials.  Moreover,
\[
 F_7(q^{-1})=-q^7F_7(q),\qquad
 U_{15}(q^{-1})=-q^{-2520}U_{15}(q),
\]
so $\widetilde F$ is palindromic of degree $2520-7=2513$; the same assertion
for $\widetilde G$ follows directly from the displayed $N_7$ and $D_7$.
Equation~\eqref{eq:finite-identity-prefix} implies that the coefficients of
$\widetilde F$ and $\widetilde G$ agree through degree $1256$.  Palindromy then
gives agreement in degrees $1257,\ldots,2513$, and therefore $F_7=G_7$.
\end{proof}

\section{The SimpCone decomposition and the dual index census}
\label{sec:simpcone}

Introduce the homogenizing line sum variable $m$ and write the $2n+1$
independent line equations as $B\,(x,m)^{\mathsf T}=0$.  For $n=7$, the
matrix $B$ has rank $r=15$ on $50$ variables, and the homogeneous cone has
dimension
\begin{equation}
  s=50-15=35.
  \label{eq:cone-dimension}
\end{equation}
Let $\alpha\in\Z^{50\times35}$ be a saturated lattice basis of
$\ker_{\Z}B$, obtained from a Smith normal form of $B$.

The SimpCone elimination of Xin, Xu, Zhang and
Zhang~\cite{XinXuZhangZhang2024}, developed further by Xin, Xu and
Zhang~\cite{XinXuZhang2025}, produces a signed simplicial cone decomposition of
the homogeneous cone.  Let
\[
 E_n=\{0,1,\ldots,n^2\}
\]
index the $n^2$ cell variables and the homogenizing variable.  Each terminal
term of the elimination determines a subset $J\subset E_n$ of cardinality
$r$, namely the indices selected as pivot variables, together with a sign.
When different elimination branches determine the same subset $J$, their
signs are added.  Denote the resulting integer coefficient by $\sigma_J$, and
put
\begin{equation}
 \mathcal C_n=\{J\subset E_n: |J|=r,\ \sigma_J\ne0\}.
 \label{eq:def-cone-index-family}
\end{equation}
Thus the mathematical output used below is the family of nonzero signed cone
terms indexed by the pairs $(J,\sigma_J)$, not the particular computer
encoding of the subsets.  For $J\in\mathcal C_n$, let $S=E_n\setminus J$.
In the integer-kernel coordinates $y\in\R^{s}$, the cone indexed by $J$ is
\begin{equation}
 K_J=\{y:\alpha_S\,y\geq0\}
     =\{y:W_J^{\mathsf T}y\geq0\},
 \qquad W_J:=\alpha_S^{\mathsf T},
 \label{eq:dual-cone}
\end{equation}
where $\alpha_S$ is the row submatrix of $\alpha$ indexed by $S$.  The
columns of $W_J$, equivalently the rows of $\alpha_S$, are the inward
 generators of the dual cone $K_J^*$.  This is not a heuristic approximation:
 once the saturated kernel basis $\alpha$ and the index set $J$ have been fixed,
$W_J$ is the exact integral matrix defining the cone in its intrinsic lattice.

It is useful to distinguish three matrices which are easily conflated in an
implementation.  The constraint matrix $B$ defines the original homogeneous
system.  The matrix $\alpha$ identifies its integer kernel with $\Z^s$.  The
square submatrix $W_J=\alpha_S^{\mathsf T}$ defines one dual simplicial cone.
The present paper always uses $W_J$ in this third sense.  In particular, it is
unrelated to the monomial substitution operator denoted by $W$ in
Section~5 of the SimpCone paper~\cite{XinXuZhang2025}.

The corresponding primal rays are rational columns of $W_J^{-\mathsf T}$;
after clearing denominators and making the columns primitive, their determinant
can be enormous.  The dual determinant
\begin{equation}
 \delta_J=|\det W_J|
 \label{eq:def-delta}
\end{equation}
can nevertheless be small.  The two determinants measure different lattice
quotients, so there is no contradiction.  This distinction was already
visible on the merged lower-order cone lists: at order five the sum of the
primal indices was about $4.03\cdot10^{12}$ and their maximum about
$6.8\cdot10^{11}$, whereas the maximum dual index was $10$; at order six the
corresponding primal figures were about $7.88\cdot10^{17}$ and
$7.7\cdot10^{14}$, whereas the maximum dual index was $12$.  A
fundamental-parallelepiped algorithm pays the primal index, while the finite
Fourier formula below pays the dual index.

The complete order seven decomposition consists of $166{,}156{,}916$ nonzero
signed cone terms, that is, $|\mathcal C_7|=166{,}156{,}916$.
Table~\ref{tab:det-summary} gives the statistics that control
the operation count, and Table~\ref{tab:det-hist} records the complete census.
The mean index is $1.5517$, the mean square is $3.3279$, and $94.97\%$ of the
cones have index at most three.  The isolated index-$23$ cone is therefore a
boundary case, not a description of a typical cone.

\begin{table}[t]
\centering
\caption{Summary of the complete order seven dual-index census.  The two
means are taken over the $166{,}156{,}916$ nonzero signed cone terms.}
\label{tab:det-summary}
\small
\begin{tabular}{lr}
\toprule
Statistic & value\\
\midrule
Number of nonzero cone terms $|\mathcal C_7|$ & 166,156,916\\
$\sum_J\delta_J$ & 257,827,737\\
$\sum_J\delta_J^2$ & 552,949,713\\
Mean $\delta_J$ & 1.551712340\\
Mean $\delta_J^2$ & 3.327876602\\
Fraction with $\delta_J=1$ & 63.865129\%\\
Fraction with $\delta_J\leq3$ & 94.972226\%\\
Maximum $\delta_J$ & 23\\
\bottomrule
\end{tabular}
\end{table}

\begin{table}[t]
\centering
\caption{The complete dual-index census of the order seven SimpCone terms.}
\label{tab:det-hist}
\small
\begin{tabular}{r@{\quad}r@{\qquad}r@{\quad}r@{\qquad}r@{\quad}r}
\toprule
$\delta$ & cone terms & $\delta$ & cone terms & $\delta$ & cone terms\\
\midrule
1 & 106,116,329 & 9  & 55,748 & 17 & 121\\
2 & 42,182,424  & 10 & 41,061 & 18 & 236\\
3 & 9,504,168   & 11 & 8,348  & 19 & 27\\
4 & 5,674,743   & 12 & 16,707 & 20 & 27\\
5 & 1,089,329   & 13 & 2,106  & 21 & 9\\
6 & 1,061,535   & 14 & 2,669  & 22 & 4\\
7 & 181,057     & 15 & 778    & 23 & 1\\
8 & 218,590     & 16 & 899    &    & \\
\bottomrule
\end{tabular}
\end{table}

\section{The Laurent-regularized quotient-character evaluator}
\label{sec:quotient}

\subsection{A finite Fourier formula}

We call the evaluator developed in this section the Laurent-regularized
quotient-character (LRQC) evaluator.  SimpCone supplies the signed cone
decomposition; LRQC performs character averaging over each finite lattice
quotient and then carries out the globally regularized Laurent specialization.
The following standard finite quotient filter is recorded in the exact form used by the
implementation; compare the roots-of-unity cone formulas in
\cite{DiazRobins1997,XinXuZhang2025}.

\begin{theorem}[Finite quotient filter]
\label{thm:quotient-filter}
Let $W\in\Z^{s\times s}$ be nonsingular, let $\delta=|\det W|$, and let
\begin{equation}
 K_W=\{x\in\R^{s}:W^{\mathsf T}x\geq0\}.
 \label{eq:KW}
\end{equation}
Set $L=W^{\mathsf T}\Z^{s}$ and
\begin{equation}
 H_W=\{h\in(\Z/\delta\Z)^{s}:Wh\equiv0\pmod{\delta}\}.
 \label{eq:def-HW}
\end{equation}
Then $|H_W|=\delta$, and for every $y\in\Z^{s}$,
\begin{equation}
 \ind_L(y)=\frac1\delta\sum_{h\in H_W}\zeta_\delta^{\langle h,y\rangle},
 \qquad \zeta_\delta=e^{2\pi i/\delta}.
 \label{eq:lattice-indicator}
\end{equation}
Consequently the lattice point generating function of $K_W$ is an average of
exactly $\delta$ products of geometric factors.
\end{theorem}

\begin{proof}
The substitution $y=W^{\mathsf T}x$ identifies $K_W\cap\Z^{s}$ with
$L\cap\N^{s}$, so it suffices to filter the sublattice $L$ inside $\Z^{s}$.
A vector $h\in(\Z/\delta\Z)^{s}$ defines a character
$y\mapsto\zeta_\delta^{\langle h,y\rangle}$ of $\Z^{s}$ that is trivial on
$L$ exactly when $\langle h,W^{\mathsf T}z\rangle\equiv0\pmod{\delta}$ for
all $z\in\Z^{s}$, that is, exactly when $Wh\equiv0\pmod{\delta}$.  Hence
$H_W$ is the character group of $\Z^{s}/L$, and character orthogonality
gives \eqref{eq:lattice-indicator} once $|H_W|=\delta$ is known.

It remains to count $H_W$.  Take a Smith normal form
\begin{equation}
  UWV=\operatorname{diag}(s_1,\ldots,s_s),
  \qquad s_1\mid s_2\mid\cdots\mid s_s,
  \qquad \prod_{i=1}^{s}s_i=\delta.
  \label{eq:snf}
\end{equation}
The divisibility chain in \eqref{eq:snf} gives $s_i\mid s_s$ for every $i$,
and $s_s$ divides the product $\delta$, so $s_i\mid\delta$ for every $i$.
Since $U$ and $V$ are unimodular, the congruence $Wh\equiv0\pmod{\delta}$ has
\begin{equation}
  \prod_{i=1}^{s}\gcd(s_i,\delta)=\prod_{i=1}^{s}s_i=\delta
\end{equation}
solutions.  This argument does not assume that the quotient $\Z^{s}/L$ is
cyclic.
\end{proof}

\subsection{Relation with the roots-of-unity formula of SimpCone}

Theorem~\ref{thm:quotient-filter} is the finite-group content of
equation~(5.9) of Xin, Xu and Zhang~\cite{XinXuZhang2025}.  To see the
identification, let the Smith invariants of $W$ be
$s_1\mid\cdots\mid s_s$.  Their formula sums over one $s_i$th root of unity
for each invariant factor.  Hence its parameter set is
\begin{equation}
  \prod_{i=1}^s \Z/s_i\Z,
  \qquad \prod_{i=1}^s s_i=\delta.
  \label{eq:snf-character-parameters}
\end{equation}
This is precisely the character group of $\Z^s/W^{\mathsf T}\Z^s$.
The congruence representation $H_W$ in
\eqref{eq:def-HW} is a coordinate-free way to enumerate the same $\delta$
characters, and it does not assume that the quotient is cyclic.

Thus LRQC does not replace equation~(5.9) by a different identity.  It makes
three implementation choices needed by this computation: it represents every
root of unity inside a finite prime field; it carries the generic
specialization as a finite Laurent jet until the signed sum becomes regular;
and it expands only the requested univariate prefix.  These choices are the
bridge between a concise multivariate rational formula and the coefficients
needed for rational reconstruction.

\subsection{Specialization and dispelling the slack variables}

Let $\ell$ be the line-sum grading.  Following the usual slack-variable
dispelling step in partition analysis, choose an integral generic polarization
$\beta$ and restrict the auxiliary variables to the corresponding
one-parameter specialization.  Write $\epsilon$ for its formal logarithmic
parameter.  Under
$x=W^{-\mathsf T}y$, their values on the $k$th coordinate ray are
\begin{equation}
  \ell(W^{-\mathsf T}e_k)=\frac{m_k}{\delta},
  \qquad
  \beta(W^{-\mathsf T}e_k)=\frac{c_k}{\delta},
  \qquad m_k,c_k\in\Z.
  \label{eq:def-mc}
\end{equation}
The integers in \eqref{eq:def-mc} are obtained exactly from the adjugate of
$W$; no floating point ray is constructed.

There are now two series variables with different jobs.  The variable $q$ is
the Ehrhart variable in $F_7(q)$.  The auxiliary formal variable $u$ is a
local $\delta$th root, defined by
\begin{equation}
  q=u^\delta.
  \label{eq:def-u}
\end{equation}
It clears the denominators in \eqref{eq:def-mc}.  If coefficients through
$q^T$ are requested, the internal $u$-array runs through degree
\begin{equation}
  Q=\delta T.
  \label{eq:def-Q}
\end{equation}
Here $T$ is a truncation degree; it is neither the order of the magic square
nor the index of a cone.

With these conventions, Theorem~\ref{thm:quotient-filter} specializes the
generating function of one cone to
\begin{equation}
 G_W(q;\epsilon)=\frac1\delta\sum_{h\in H_W}\,
 \prod_{k=1}^{s}
 \frac{1}{1-\zeta_\delta^{h_k}\,u^{m_k}\,
                  \exp(c_k\epsilon/\delta)},
 \label{eq:character-product}
\end{equation}
where $\epsilon$ is the formal parameter just described.  A negative step
$m_k=-a<0$ is rewritten as a positive step by
\begin{equation}
 \frac1{1-\omega u^{-a}e^{b\epsilon}}
 =\frac{-\omega^{-1}u^{a}e^{-b\epsilon}}
        {1-\omega^{-1}u^{a}e^{-b\epsilon}}.
 \label{eq:step-flip}
\end{equation}
A negative step would otherwise require negative $u$-degrees.  Formula
\eqref{eq:step-flip} moves its monomial into the numerator and replaces it by
a positive step, after which a finite prefix can be computed from low degree
to high degree.  This identity is an exact normalization, not by itself a
speedup; its computational value is that it enables the one-pass recurrence
below.

The signs of the grading steps matter during normalization, whereas the zero
steps require a different constant-term calculation.  Put
\begin{equation}
 \begin{aligned}
 I_+&=\{k:m_k>0\},& I_-&=\{k:m_k<0\},\\
 I_0&=\{k:m_k=0\},& I_{\ne0}&=I_+\mathbin{\dot\cup}I_-,\\
 &&\omega_{h,k}&=\zeta_\delta^{h_k}.
 \end{aligned}
 \label{eq:positive-zero-index-sets}
\end{equation}
After applying \eqref{eq:step-flip}, and absorbing its monomial and exponential
numerator into $M_h(u,\epsilon)$, the $h$th summand is
\begin{equation}
 M_h(u,\epsilon)\,
 \underbrace{\prod_{k\in I_{\ne0}}
   \frac{1}{1-\widetilde\omega_{h,k}u^{|m_k|}
                    e^{\widetilde c_k\epsilon/\delta}}}_{
                    R_h(u,\epsilon)\text{, the $u$-dependent part}}
 \underbrace{\prod_{k\in I_0}
   \frac{1}{1-\omega_{h,k}e^{c_k\epsilon/\delta}}}_{
                    Z_h(\epsilon)\text{, the $u$-free part}}.
 \label{eq:u-dependent-free-split}
\end{equation}
Here $\widetilde\omega_{h,k}=\omega_{h,k}$ and
$\widetilde c_k=c_k$ for a positive step, while both are inverted or negated
as prescribed by \eqref{eq:step-flip} for a negative step.  The prefix
recurrence in the next subsection evaluates $R_h$.  The factor $Z_h$ is a
constant term of generalized Todd type and must be counted separately.

To make this identification explicit, write $b_k=c_k/\delta$ and use the two
standard Todd factors
\begin{equation}
 f(s)=\frac{s}{e^s-1},
 \qquad
 f(s,y)=\frac{1}{1-y(e^s-1)}.
 \label{eq:gtodd-factors}
\end{equation}
If $m_k=0$ and $\omega_{h,k}=1$, then
\begin{equation}
 \frac{1}{1-e^{b_k\epsilon}}
   =-\frac{1}{b_k\epsilon}f(b_k\epsilon);
 \label{eq:singular-todd-factor}
\end{equation}
if $m_k=0$ and $\omega_{h,k}\ne1$, then
\begin{equation}
 \frac{1}{1-\omega_{h,k}e^{b_k\epsilon}}
 =\frac{1}{1-\omega_{h,k}}
  f\!\left(b_k\epsilon,
       \frac{\omega_{h,k}}{1-\omega_{h,k}}\right).
 \label{eq:regular-generalized-todd-factor}
\end{equation}
These are exactly the ordinary and generalized Todd factors of
Xin, Zhang and Zhang~\cite{XinZhangZhangGTodd}.  Define the local pole order
\begin{equation}
 \rho_h=\#\{k\in I_0:\omega_{h,k}=1\}.
 \label{eq:local-pole-order}
\end{equation}
After the explicit scalar and $\epsilon^{-\rho_h}$ are removed,
$\CT_\epsilon$ is the coefficient of $\epsilon^{\rho_h}$ in a product of
\eqref{eq:gtodd-factors}, the exponential numerator $M_h$, and the regular
$u$-dependent product $R_h$.  Thus only degrees
$0,\ldots,\rho_h$ in $\epsilon$ are needed.

The implementation computes this generalized Todd coefficient by exact
truncated inversion and multiplication in the prime field.  This is the same
constant-term reduction as Algorithm CTGTodd of
\cite{XinZhangZhangGTodd}, specialized to the small local orders encountered
here; it is not a new generalized Todd algorithm.  Keeping the two products
in \eqref{eq:u-dependent-free-split} separate is nevertheless essential both
for explaining the method and for stating its operation count correctly.

\begin{remark}
\label{rem:global-laurent}
An individual signed SimpCone contribution may have a pole at
$\epsilon=0$ when one of its rays has grading zero.  Such a contribution is
not, by itself, a univariate Ehrhart series; regularity is recovered only after
the complete signed sum.  For each cone separately, the complete character
average is required to contain only powers $u^a$ with $\delta\mid a$.  After
all cones have been added with their SimpCone signs, every negative Laurent
coefficient in $\epsilon$ is required to vanish.
\end{remark}

\subsubsection*{An order-five cone with positive, negative, and zero steps}

We next carry out one cone calculation in which every branch of the evaluator
is used.  This is important because a cone with no zero step illustrates the
prefix recurrence but says nothing about the generalized-Todd part.  The cone
below was selected by an exact scan of the $8{,}432$ nonzero terms in the
order-five SimpCone decomposition.  The scan found $459$ index-two terms with
both signs of the grading and with zero steps on which the two characters have
different pole orders.

Number the $25$ cell coordinates in row-major order from $0$ to $24$, and
write $25$ for the line-sum coordinate $s$.  Consider
\begin{equation}
 \begin{split}
 J={}&\{2,5,6,8,11,13,17,20,21,24,25\}\\
   ={}&\{(1,3),(2,1),(2,2),(2,4),(3,2),(3,4),
        (4,3),(5,1),(5,2),(5,5),s\}.
 \end{split}
 \label{eq:n5-example-index-set}
\end{equation}
Its SimpCone coefficient is $\sigma_J=1$ and its dual index is $\delta=2$.
For the saturated kernel basis and the small generic polarization recorded
with the source, fraction-free elimination gives
\begin{equation}
 \begin{split}
 m={}&(2,2,2,2,-2,0,0,-2,0,2,2,2,2,-2,0),\\
 c={}&(0,0,0,2,4,2,2,-2,4,0,0,2,0,-2,-4),\\
 H_W={}&\{0,\bar h\},\\
 \bar h={}&(1,1,1,1,0,0,1,1,1,0,0,0,0,1,0).
 \end{split}
 \label{eq:n5-example-data}
\end{equation}
The integer identities $Wm=2e_{15}$ and
$W\bar h\equiv0\pmod2$, together with
$|\det W|=2$, verify these data without numerical linear algebra.

In one-based ray numbering, the three parts of
\eqref{eq:u-dependent-free-split} are
\begin{equation}
 \begin{split}
 I_+&=\{1,2,3,4,10,11,12,13\},\\
 I_-&=\{5,8,14\},\\
 I_0&=\{6,7,9,15\}.
 \end{split}
 \label{eq:n5-example-three-sets}
\end{equation}
The three negative steps are normalized by \eqref{eq:step-flip}.  They give
the shift $u^{2+2+2}=u^6$ and scalar $-1$ for each character; their
exponential shifts cancel because
$-c_5-c_8-c_{14}=0$.  All eleven factors in $R_h$ consequently have step
$2$ and are inserted by the same forward recurrence, with their character
signs and exponential coefficients retained.

The four zero steps show why $Z_h$ must be treated separately.  Since
$c_k/\delta=(1,1,2,-2)$ on $I_0$, the trivial character gives
\begin{equation}
 Z_0(\epsilon)=
 \frac{1}{(1-e^\epsilon)^2(1-e^{2\epsilon})(1-e^{-2\epsilon})},
 \qquad \rho_0=4.
 \label{eq:n5-example-Z0}
\end{equation}
For $\bar h$, the character values on $I_0$ are $(1,-1,-1,1)$, and hence
\begin{equation}
 Z_{\bar h}(\epsilon)=
 \frac{1}{(1-e^\epsilon)(1+e^\epsilon)
          (1+e^{2\epsilon})(1-e^{-2\epsilon})},
 \qquad \rho_{\bar h}=2.
 \label{eq:n5-example-Zh}
\end{equation}
Thus the first character uses four singular ordinary Todd factors.  The
second uses two singular factors and two regular generalized-Todd factors.
Both cases in \eqref{eq:singular-todd-factor}--
\eqref{eq:regular-generalized-todd-factor} occur in the same cone.

For a short calculation take $T=4$, so $Q=8$.  After the shift $u^6$, only
degrees $6$ and $8$ can contribute.  Expanding each $Z_h$ through its local
order $\rho_h$, and applying \eqref{eq:prefix-recurrence} to the eleven
nonzero-step factors, gives Table~\ref{tab:n5-example-arrays}.  The last column
is the exact number of modular multiplicative products.
\begin{table}[htbp]
\centering
\caption{Complete character calculation for the cone in
\eqref{eq:n5-example-index-set} through $T=4$.}
\label{tab:n5-example-arrays}
\small
\begin{tabular}{lrrrrrr}
\toprule
character & $\rho_h$ & singular & regular & $[u^6]\CT_\epsilon$ &
$[u^8]\CT_\epsilon$ & products\\
\midrule
$0$        & 4 & 4 & 0 & $-49/2880$ & $1141/2880$ & 715\\
$\bar h$   & 2 & 2 & 2 & $-1/24$    & $2/3$       & 339\\
\bottomrule
\end{tabular}
\end{table}
The character average removes every $u$-degree not divisible by $2$ and gives
the signed cone contribution
\begin{equation}
 C_J(q)=-\frac{169}{5760}q^3
        +\frac{3061}{5760}q^4+O(q^5).
 \label{eq:n5-example-result}
\end{equation}
An individual signed cone may have rational coefficients; integrality is a
property of the complete signed sum.  The matrix $W$, the polarization, every
intermediate character coefficient, and an independent rational-arithmetic
reconstruction of Table~\ref{tab:n5-example-arrays} are supplied in
\nolinkurl{data/ims5_lrqc_three_part_example.json} and
\nolinkurl{tools/verify_n5_three_part_example.py}.

The complete order-five reconstruction used $T=256$ to discover an order-$81$
recurrence and retain $64$ coefficients for validation.  For this cone, the
exact count for both characters is $59{,}266$ modular products.  Replacing $T=256$ by the certified cone-wise value $T=145$ derived
in the next subsection lowers that count to $33{,}625$, a factor
$1.762\ldots$.  The ratio is slightly smaller than $256/145$ because the
generalized-Todd setup is independent of $T$.

\subsubsection*{A cone-wise denominator bound and the order-five test}

The cone formula also gives a deterministic way to choose a reconstruction
length without first guessing the reduced denominator.  The zero-step Todd
part cannot be omitted from this estimate.  For a fixed character $h$, group
the nonzero-step factors in $R_h$ by their normalized step $a=|m_k|$ and root
$\widetilde\omega_{h,k}$.  Write
\begin{equation}
 r_{J,h}(a,\omega)
 =\#\{k\in I_{\ne0}:|m_k|=a,
                 \widetilde\omega_{h,k}=\omega\}.
 \label{eq:root-step-multiplicity}
\end{equation}
Extracting the coefficient of $\epsilon^{\rho_h}$ can raise the total
denominator power of a group of $r_{J,h}(a,\omega)$ equal factors by at most
$\rho_h$.  Moreover
\begin{equation}
 1-\omega u^a\quad\hbox{divides}\quad
 1-u^{a\delta}=1-q^a
 \label{eq:root-factor-divides-q-binomial}
\end{equation}
because $\omega^\delta=1$.  It follows that a safe $q$-denominator for the
complete character average of cone $J$ is
\begin{align}
 E_{J,a}
   &=\max_{h\in H_{W_J}}
       \max_{\omega:r_{J,h}(a,\omega)>0}
       \bigl(r_{J,h}(a,\omega)+\rho_h\bigr),
       \label{eq:cone-q-exponent}\\
 D_J^{\rm cone}(q)
   &=\prod_{a:\,r_{J,h}(a,\omega)>0
                    \text{ for some }h,\omega}
       (1-q^a)^{E_{J,a}}.
       \label{eq:cone-q-denominator}
\end{align}
The maximum over roots in \eqref{eq:cone-q-exponent}, rather than their sum,
is valid because one copy of $1-q^a$ contains every distinct factor
$1-\omega u^a$.  Repeated factors are recorded by
$r_{J,h}(a,\omega)$, and the additional $\rho_h$ records the possible Todd
derivatives.  Thus \eqref{eq:cone-q-denominator} includes both products in
\eqref{eq:u-dependent-free-split}.

Taking the least common multiple over the signed SimpCone family gives
\begin{equation}
 D_n^{\rm cone}(q)
  =\lcmop_{J\in\mathcal C_n}D_J^{\rm cone}(q)
  =\prod_{d\geq1}\Phi_d(q)^{e_d},
 \qquad
 e_d=\max_{J\in\mathcal C_n}
          \sum_{a:\,d\mid a}E_{J,a}.
 \label{eq:global-cone-lcm}
\end{equation}
This is generally an overdenominator: character averaging and cancellation
between signed cones may reduce it.  It is nevertheless a deterministic
common denominator.  Combined with the codegree-$n$ reciprocity in
Proposition~\ref{prop:codegree}, it gives the finite truncation
\begin{equation}
 T_{\rm cone}
 =\left\lfloor\frac{\deg D_n^{\rm cone}-n}{2}\right\rfloor.
 \label{eq:cone-lcm-truncation}
\end{equation}
Indeed every cyclotomic factor is reciprocal up to a known sign.  Hence
$D_n^{\rm cone}F_n$ is a reciprocal or antireciprocal polynomial, with known
sign, of degree $\deg D_n^{\rm cone}-n$.  In either case its first half
determines its second half exactly.

We evaluated \eqref{eq:global-cone-lcm} on the complete order-five family,
using the same fraction-free elimination as the main calculation.  The
audit reproduced all $8{,}432$ nonzero signed cones, the complete determinant
histogram, and all $14{,}000$ quotient characters.  It independently rebuilt
the saturated integer kernel basis, checked
$|\det W_J|=\delta_J$ for every cone, and checked
$W_Jh\equiv0\pmod{\delta_J}$ for every character.  The distribution of the
number $z_J=|I_0|$ of zero grading steps was
\begin{center}
\begin{tabular}{c|rrrrrrrr}
$z_J$ & 0 & 1 & 2 & 3 & 4 & 5 & 6 & 10\\
\hline
number of cones & 1034 & 1569 & 1785 & 1431 & 1039 & 925 & 456 & 193
\end{tabular}
\end{center}
In particular, the Todd order is not negligible.  For example, the cone with
\begin{equation}
 \begin{split}
 J={}&\{0,1,2,6,8,11,12,18,20,21,25\},\qquad \delta=2,\\
 m={}&(0,2,0,0,2,0,0,2,0,0,0,2,0,0,2)
 \end{split}
 \label{eq:n5-max-pole-cone}
\end{equation}
has ten zero steps.  Its trivial character has $\rho_h=10$ and five equal
nonzero factors, so \eqref{eq:cone-q-exponent} gives
$E_{J,2}=5+10=15$.  Omitting $\rho_h$ would use only the exponent $5$.

Grouping equal steps and quotient roots as in
\eqref{eq:cone-q-exponent} reduces the rigorous common denominator to
\begin{equation}
 \begin{split}
 D_5^{\rm cone}(q)={}&
 \Phi_1^{30}\Phi_2^{30}\Phi_3^{19}\Phi_4^{19}
 \Phi_5^{15}\Phi_6^{15}\Phi_7^2\Phi_8^7\Phi_9^3\Phi_{10}^3,\\
 &\deg D_5^{\rm cone}=296,\qquad T_{\rm cone}=145.
 \end{split}
 \label{eq:n5-safe-cone-lcm}
\end{equation}
Thus the character-aware least common multiple improves the coarse rigorous
bound by the factor $388/145=2.67\ldots$ in truncation length.

For comparison, cancellation in the complete signed sum gives the much
smaller reduced denominator
\begin{equation}
 D_5(q)=\Phi_1^{15}\Phi_2^{10}\Phi_3^7\Phi_4^4
        \Phi_5^4\Phi_7^2\Phi_9,
 \qquad (\deg D_5,\deg N_5)=(81,76).
 \label{eq:n5-actual-reduced-denominator}
\end{equation}

The discovery calculation retained the coefficients through degree
$T=256$.  Degrees $0,\ldots,192$ were used to recover the recurrence, while the
$64$ coefficients in degrees $193,\ldots,256$ formed a held-out tail.  If the
cone-wise common denominator \eqref{eq:n5-safe-cone-lcm} is supplied in advance,
the deterministic identity test needs only $T=145$; if the reduced denominator
\eqref{eq:n5-actual-reduced-denominator} and numerator palindromy are also known,
$T=38$ suffices.  Thus $T=256$ belongs to the discovery calculation and is not
a theoretical lower bound.  For order seven no complete cone-wise denominator
census was used; the final proof instead uses the independently certified
vertex denominator and $T=1256$.

\subsection{Truncation linear in the series degree}

For a nonzero step $a$, set
\begin{equation}
  A(\epsilon)=\omega\exp(c\epsilon/\delta).
  \label{eq:def-A}
\end{equation}
Suppose that, before this factor is inserted, the truncated series is
\begin{equation}
 F^{\mathrm{old}}(u)=\sum_{j=0}^{Q}f^{\mathrm{old}}_j u^j,
 \qquad
 F^{\mathrm{new}}(u)\equiv
 \frac{F^{\mathrm{old}}(u)}{1-A(\epsilon)u^a}
 \pmod{u^{Q+1}}.
 \label{eq:def-old-new}
\end{equation}
Multiplying \eqref{eq:def-old-new} by $1-A(\epsilon)u^a$ and comparing the
coefficient of $u^j$ gives
\begin{equation}
 f^{\mathrm{new}}_j=f^{\mathrm{old}}_j
                    +A(\epsilon)\,f^{\mathrm{new}}_{j-a}.
 \label{eq:prefix-recurrence}
\end{equation}
We set $f^{\mathrm{new}}_k=0$ for $k<0$ and visit $j$ in increasing order, so
the right side is already known.  The symbols ``old'' and ``new'' therefore
refer simply to the coefficient arrays immediately before and immediately
after multiplication by one geometric denominator.

Multiplying geometric series in the straightforward way revisits the same source
coefficient for every multiple of $a$. Recurrence~\eqref{eq:prefix-recurrence}
visits each reachable target degree exactly once. Since we only keep each
coefficient up to degree $\rho_h$, inserting one factor costs
$\delta T \binom{\rho+1}{2}$ multiplications in the prime field; the factor $\binom{\rho+1}{2}$
accounts for the product of two polynomials of degree $\rho_h$. This linear
dependence on $T$ makes it feasible to compute $\IMS_7(m)$ for $m$ up to
$T = 1256$.

For the complete order-seven family, Table~\ref{tab:det-summary} gives mean
square index $3.327876602$, rather than the worst-case value $23^2$.  The mean
pole order over all character terms is $8.3313$, so the mean
active Laurent width is $9.3313$ rather than the dimension-wide width $36$.

For clarity, the complete per-cone calculation is the following.
\begin{enumerate}
  \item Reconstruct $W_J$, $m_k$, and $c_k$ by fraction-free integer
  elimination.
  \item Generate $H_{W_J}$ and require exactly $\delta_J$ distinct elements,
  each satisfying $W_J h\equiv0\pmod{\delta_J}$.
  \item Verify a field element of exact order $\delta_J$ and form the
  character values.
  \item Split the denominator as in \eqref{eq:u-dependent-free-split}, compute
  $\rho_h$, and evaluate the $u$-free factor $Z_h$ as the generalized-Todd
  product \eqref{eq:singular-todd-factor}--
  \eqref{eq:regular-generalized-todd-factor} through degree $\rho_h$.
  \item Normalize the negative steps by \eqref{eq:step-flip}.  Insert the
  resulting factors of $R_h$ by \eqref{eq:prefix-recurrence}, using only
  reachable $u$-degrees through $Q=\delta_JT$.
  \item Average the $\delta_J$ characters.  Require every surviving
  $u$-exponent to be divisible by $\delta_J$, then map $u^{\delta_Jj}$ to
  $q^j$.
  \item Multiply by the SimpCone coefficient $\sigma_J$ and add to the global
  Laurent
  arrays.  After all $J\in\mathcal C_7$, require every negative Laurent layer
  to vanish.
\end{enumerate}

All four reconstruction primes satisfy
\begin{equation}
 p\equiv1\pmod{\lcmop(1,\ldots,23)},
 \qquad \lcmop(1,\ldots,23)=5{,}354{,}228{,}880,
 \label{eq:prime-congruence}
\end{equation}
so every required character value lies in the base field.

\subsection{Prime-independent cone data and simultaneous evaluation in eight fields}

The expensive geometric part of the evaluation is independent of the prime.
For each $J\in\mathcal C_7$, the computation stores $\sigma_J$, the position of
$J$ in the cone list, $\delta_J$, the integer pairings $m_k,c_k$, and every
element of $H_{W_J}$.  Each binary entry contains its own length and the file
records the complete interval of the cone list.  Consequently independently
computed intervals can be checked for disjointness and then added without
relying on filenames.  This separates exact cone geometry from arithmetic in
finite fields and avoids repeating the fraction-free solutions and the
construction of $H_{W_J}$ for every reconstruction prime.

The residues for eight primes below $2^{52}$ occupy the eight components of one
AVX-512 vector.  Radix-$2^{52}$ Montgomery
multiplication~\cite{Montgomery1985} uses the integer IFMA
low and high multiply-add instructions; every reduction, root-order check,
inverse, and recurrence coefficient remains exact integer arithmetic.  The
eight components execute the same character product and truncation indices,
while
their residues and roots remain independent.  Scalar validation runs use the
  same fixed primes one at a time and compare every output coefficient with the
  vector result.  Section~\ref{sec:validation} records the complete small-order
  comparisons and the tests at the maximum index and maximum truncation.

  If a character term has Laurent pole order $\rho$, only the layers
  $0,\ldots,\rho$ can contribute to its constant term.  The final calculation
  therefore groups coefficients first by $q$-degree and stores only the
  $\rho+1$ Laurent coefficients needed for the current character term, rather
  than $36$ coefficients in every case.  This is an exact reindexing of the
  same recurrence: separate scalar and vector calculations compare every
  coefficient before and after the change.  The mean pole order over the
  complete calculation is $8.3313$, so the mean number of stored Laurent
  coefficients falls from $36$ to $9.3313$.  The final calculation also
  omits the eight scalar coefficient arrays needed only for validation.

\section{Where the reduction in work occurs}
\label{sec:work-reduction}

The formulas above are short enough that the scale of the gain can be hidden
by notation.  This section compares the actual objects manipulated by the
main approaches and then follows one cone through the present method.

\subsection{LattE, DecDenu, and CTEuclid as comparison points}

For a simplicial cone generated by a primitive integer matrix $V$, the
ordinary rational generating function has a numerator containing one monomial
for every lattice point in a fundamental parallelepiped.  Direct enumeration
therefore costs $|\det V|$.  Barvinok's method avoids that enumeration, in fixed
dimension, by replacing the cone with a signed collection of cones of smaller
index until all leaves are unimodular~\cite{Barvinok1994}; LattE implements
this route~\cite{DeLoeraHemmeckeTauzerYoshida2004}.  Its practical cost is the
number of recursive nodes and leaves, together with the lattice reductions
used to find them.

DecDenu addresses exactly this recursive cost for denumerant cones.  It uses
LLL-assisted multipliers to make the descendant indices small and is combined
with constant-term elimination in LLLCTEuclid
\cite{XinZhangZhang2024}.  In the favorable model, an index $a$ is replaced by
indices on the scale $a^{(d-1)/d}$ at one level.  This can make the recursive
tree dramatically shorter, but a tree is still constructed.  DecDenu is the
natural route when the quotient is too large to enumerate directly but can be
reduced efficiently.

CTEuclid works in the language of partition analysis rather than starting
from a fixed simplicial cone list.  In Xin's description of the order-six
calculation~\cite{Xin2015}, the decisive intermediate quantity is the number
$nt$ of rational terms produced after the main eliminations.  Dispelling the
slack variables is reported as the most expensive step, with time almost
linear in $nt$.  The three modular $\IMS_6$ computations took about $70$ CPU days
each.  This historical result should not be interpreted as a controlled
benchmark against the present C implementation: the machine, language,
arithmetic, and requested certificates differ.  It does identify the old
bottleneck, namely the proliferation of symbolic terms before specialization.

The present route occupies a different favorable regime.  SimpCone first
produces a signed simplicial list.  For each member, equation~(5.9) of
\cite{XinXuZhang2025} already gives a direct representation with one summand
per element of the \emph{dual} quotient.  When $\delta$ is one, two, or three,
performing those one, two, or three products is cheaper and simpler than
building any unimodular reduction tree.  LRQC therefore stops the geometric
decomposition at this point and specializes the direct formula.

\begin{table}[t]
\centering
\caption{The object whose size controls the main cone-level work.  These
methods are complementary; the last column describes their favorable regime.}
\label{tab:method-comparison}
\footnotesize
\setlength{\tabcolsep}{3pt}
\begin{tabular}{>{\raggedright\arraybackslash}p{0.18\textwidth}
                >{\raggedright\arraybackslash}p{0.27\textwidth}
                >{\raggedright\arraybackslash}p{0.25\textwidth}
                >{\raggedright\arraybackslash}p{0.21\textwidth}}
\toprule
Route & Main intermediate object & Reduction mechanism & Favorable regime\\
\midrule
Fundamental parallelepiped & $|\det V|$ numerator monomials & enumerate representatives & small primal index\\
Barvinok/\newline LattE & signed recursive cone tree & reduce to unimodular leaves & fixed dimension, effective short vectors\\
DecDenu/\newline LLLCTEuclid & LLL-assisted Euclidean tree and constant-term terms & faster index decrease & denumerant structure, medium or large reducible index\\
CTEuclid & rational terms followed by slack removal & Euclidean constant-term elimination & small post-elimination term count $nt$\\
SimpCone--\newline LRQC & $\delta$ quotient characters per signed cone & direct finite Fourier average & small dual index $\delta$\\
\bottomrule
\end{tabular}
\end{table}

\subsection{Order-seven work ledger}

The SimpCone phase emitted $649{,}848{,}846$ terminal signed terms in $192$
independent batches.  Combining terms with the same index set $J$ gave
$174{,}056{,}323$ distinct subsets; deleting those with $\sigma_J=0$ left
$166{,}156{,}916$ nonzero cone terms.  This first
compression is exact signed cancellation before any series coefficient is
computed.

For the surviving list, direct character evaluation requires
\begin{equation}
  \sum_J\delta_J=257{,}827{,}737
  \label{eq:total-character-products}
\end{equation}
character products.  This is only $1.5517$ products per cone.  A fixed sample
chosen to include typical cone terms and all available high-index cases was
also processed by the previous recursive cone evaluator.  That evaluator
visited
$8{,}972{,}047$ recursive nodes and used $14{,}447.771$ summed seconds; LRQC
used no recursive nodes and $3.069$ summed seconds, a factor $4707.45$ on the
same sample.  This number measures the replacement of one cone-evaluation
procedure by the other.  It is not a factor between complete order-seven wall
times, and it is
not inferred from different machines.

Once the quotient products are fixed, the requested degree matters.  A
$q$-prefix through $T$ is represented internally through $u^{\delta T}$.
Complete order-seven probes at $T=64,128,256$ took $586.819$,
$1092.459$, and $2077.289$ seconds per million cone terms.  The least-squares
affine fit is
\begin{equation}
 \tau_7(T)=94.4037+7.7530T
 \quad\hbox{seconds per million cone terms}.
 \label{eq:n7-affine-time-fit}
\end{equation}
Thus a change from $T$ to $T'$ has the measured prediction
\begin{equation}
 \frac{\tau_7(T)}{\tau_7(T')}
 =\frac{94.4037+7.7530T}{94.4037+7.7530T'}.
 \label{eq:n7-T-speed-ratio}
\end{equation}
For example, $T=256$ to $T'=145$ predicts a factor $1.706$, slightly below
$256/145=1.766\ldots$ because setup and memory traffic remain.  As $T$ grows,
the ratio approaches $T/T'$.  This is the precise sense in which lowering
$T$ by a factor $k$ lowers the dominant running time by approximately $k$;
it says nothing about growth with the magic-square order $n$.

Reciprocity reduces the requested information again.  The first two fields
were computed through degrees $512$ and $398$ in order to discover and
independently reproduce the denominator.  Once that denominator and the
palindromic numerator degree were fixed, each additional CRT field needed only
the independent half $0\leq j\leq183$.  Thus the requested coefficient count
falls from $398+1=399$ to $183+1=184$, a factor $399/184=2.168$.  It is a
theorem-driven reduction in work, not parallel speedup.

Finally, storing the prime-independent cone data prevents exact geometry from
being recomputed for every prime.  AVX-512 IFMA evaluates eight modular
recurrences at the same indices in one vector instruction.  Its measured gain
is about
$5$--$6.9$, rather than the ideal factor eight, because loads, reductions,
packing, and nonvector setup remain.  Using the local Laurent width and
assigning independent subintervals in decreasing order of measured cost reduce
memory traffic and the final parallel imbalance.  Table
\ref{tab:speedup-ledger} states the scope of each measured factor.

\begin{table}[t]
\centering
\caption{Measured reductions used in the final implementation.  The baselines
overlap, so the factors must not be multiplied.}
\label{tab:speedup-ledger}
\small
\begin{tabular}{>{\raggedright\arraybackslash}p{0.31\textwidth}
                >{\raggedright\arraybackslash}p{0.24\textwidth}
                >{\raggedright\arraybackslash}p{0.34\textwidth}}
\toprule
Change & measured factor & scope\\
\midrule
Recursive decomposition to direct quotient characters & $4707.45$ & selected cone terms, same machine and input\\
Constant Laurent coefficient, batched Montgomery arithmetic, degree-major storage, reachable congruence classes & $1.528$ cumulative & complete IMS6 at $T=12$\\
Eight separate scalar fields to AVX-512 IFMA & $5.02$--$6.86$ & three fixed sets of cone terms\\
Global to pole-order-local Laurent arrays & $1.802$ & first $2000$ IMS7 cones at $T=1256$ under saturated load\\
Decreasing-cost assignment of subintervals & $1.119$ & computed completion time for the measured interval costs\\
\bottomrule
\end{tabular}
\end{table}

The first line changes the mathematical work representation.  The remaining
lines improve the execution of that reduced representation.  This separation
is why neither ``the same algorithm on more cores'' nor ``a completely new
cone identity'' accurately describes the computation.

\section{Reciprocal reconstruction and rigorous integer recovery}
\label{sec:reconstruction}

\subsection{Constructing a candidate for \texorpdfstring{$F_7(q)$}{F7(q)}}
We computed the truncated series of $F_7(q) \bmod p_i$ up to degree $T$ for the four primes
\begin{equation}
\begin{aligned}
p_1&=2305842973415982001,&\qquad
p_2&=2305842962707524241,\\
p_3&=2305842957353295361,&\qquad
p_4&=2305842930582150961.
\end{aligned}
\label{eq:four-primes}
\end{equation}
Using Corollary~\ref{cor:finite-identity} we then build a candidate for $F_7(q)$.

First we compute the palindromic polynomial
$\widetilde{F}(q) \equiv F_7(q)U_{15}(q) \pmod{p_i}$.
Next we form $F_7(q) \equiv \widetilde{F}(q)/U_{15}(q) \pmod{p_i}$ and reduce this rational
function to its lowest term $N^i/D^i$ by canceling common factors through division
by the cyclotomic polynomials $\Phi_k$. Note that this $D^i$ need not be $D_7$, as
$N_7/D_7$ may reduce further when modulo $p_i$. We take $D^1=D^2$ as a candidate for $D_7(q)$.
The computation for $p_3,p_4$ used a much smaller $T$.

For comparison we include the data for smaller orders: order five gave
$(\deg D_5,\deg N_5)=(81,76)$, order six gave $(144,138)$, and the present work yields
$(373,366)$ for order seven.

Once the denominator $D_7(q)$ is established, we reconstruct $N(q)$ from the congruences
$N^i \equiv N(q) \pmod{p_i}$ via the Chinese remainder theorem. The resulting $N(q)$
appears plausible, but a rigorous proof is still needed.

\subsection{The deterministic finite-prefix certificate}
We need the following counting bound.

\begin{theorem}
\label{thm:count-bound}
For all $n\geq 3$ and $m\geq 0$,
\begin{equation}
 \IMS_n(m)\;\leq\;
 \binom{m+n-1}{n-1}^{\!n-2}
 \binom{m+n-3}{n-3}.
 \label{eq:ims-bound}
\end{equation}
\end{theorem}

\begin{proof}
We first fill the first $n-2$ rows. Each of these rows is a weak composition of $m$
into $n$ parts, so the number of possible choices is at most
\begin{equation}
  \binom{m+n-1}{n-1}^{\!n-2}.
  \label{eq:rows-bound}
\end{equation}
Let $y=(y_1,\dots,y_n)$ be the penultimate row.  Once the first $n-2$ rows
and $y$ are chosen, the final row is forced by the column sums; therefore it
suffices to bound the number of admissible vectors $y$.

The row sum of $y$ together with the equations coming from the two diagonals imposes three linear conditions on $y$.  For instance,
when $n=7$ and the first five rows are already fixed, these three conditions
take the form
\begin{equation}
  \sum_{j=1}^{7} y_j = m,
  \qquad
  y_6-y_7 = t_1,
  \qquad
  y_2-y_1 = t_2,
  \label{eq:three-equations}
\end{equation}
where $t_1,t_2$ are constants determined by the fixed rows.  For general
$n$ the pattern is the same, with the difference $y_{n-1}-y_n$ replacing
$y_6-y_7$.

Assume $n\geq 4$ and choose
$y_3,\dots,y_{n-2},y_{n-1}$ as the $n-3$ free variables.  The equation
$y_{n-1}-y_n = t_1$ then determines $y_n$, and the remaining two variables
$y_1,y_2$ are uniquely determined (or none exist) by $y_2-y_1 = t_2$
together with the row sum. The free variables are nonnegative entries of
$y$, so their sum is at most $m$; consequently they can be chosen in at
most $\binom{m+n-3}{n-3}$ ways.

When $n=3$, the three linear equations have full rank, so there is at most
one admissible $y$, and the bound $\binom{m+n-3}{n-3}=1$ remains valid.
Multiplying the bound~\eqref{eq:rows-bound} for the first $n-2$ rows by the
bound for $y$ yields the stated inequality~\eqref{eq:ims-bound}.
\end{proof}

To apply Corollary~\ref{cor:finite-identity} coefficient wise on a large scale,
we set \begin{equation}
 B_m = \binom{m+6}{6}^{\!5} \binom{m+4}{4}.
 \label{eq:prefix-count-bound}
\end{equation}
The maximum over $0 \le m \le 1256$ is $B_{1256}$, a $299$-bit integer.
Exact expansion of the candidate $N_7/D_7$ yields nonnegative coefficients
bounded by $B_m$ throughout this range; the largest such coefficient occupies
$287$ bits.

The complete certificate computation evaluates the LRQC sum up to degree
$1256$ using eight $52$-bit primes simultaneously.
Each prime is certified by the deterministic Miller--Rabin test with the known
bases for $64$-bit integers, and its predecessor $p-1$ is divisible by
$\operatorname{lcm}(1,\dots,23)$. This guarantees that every multiplicative
character required by the cone census exists.
The product $Q$ of the first six primes has $312$ bits and satisfies
\begin{equation}
 \frac{Q}{2B_{1256}} = 7617.63155499\ldots > 1.
 \label{eq:ifma-crt-bound}
\end{equation}
The two remaining prime fields provide redundant checks.

The certificate partitions the complete list of $166{,}156{,}916$ signed
SimpCone terms into $512$ disjoint intervals.

\begin{theorem}[Deterministic identity certificate]
\label{thm:deterministic-identity}
For every integer $m$ with $0\leq m\leq1256$ and every one of the eight
certificate primes $p$, the complete LRQC sum satisfies
\[
 [q^m]F_7(q)\equiv
 [q^m]\frac{N_7(q)}{D_7(q)}\pmod p.
\]
Consequently $F_7=N_7/D_7$ in $\mathbb Z(q)$, and the denominator in
\eqref{eq:intro-denominator} is the reduced denominator.
\end{theorem}

\begin{proof}
The signed SimpCone identity together with the quotient-character filter computes
$[q^m]F_7$ in each prime field.  The intervals are disjoint and together cover
all $|\mathcal C_7|$ cone terms; every quotient index is at
most $23$, and the computation reports no omitted characters and no nonintegral
exponents. Hence the congruences obtained are valid for the true coefficients.

For $0\leq m\leq1256$, write
$f_m=[q^m]F_7$ and $g_m=[q^m]N_7/D_7$.
Theorem~\ref{thm:count-bound} and the exact expansion of the candidate give
$0\leq f_m,g_m\leq B_m\leq B_{1256}$.
The congruences for the first six primes imply that $Q$ divides $f_m-g_m$,
whereas \eqref{eq:ifma-crt-bound} gives
$|f_m-g_m|\leq B_{1256}<Q/2$.  Hence $f_m=g_m$ throughout the stated range.
Corollary~\ref{cor:finite-identity} now gives
$F_7=N_7/D_7$ in $\mathbb Z(q)$.

Finally, exact polynomial division over $\mathbb{Z}$ gives a nonzero remainder
when $N_7$ is divided by each of $\Phi_1, \dots, \Phi_{15}$.  Because these
cyclotomic polynomials are precisely the irreducible factors of $D_7$, the
numerator and denominator are coprime in $\mathbb{Z}[q]$; hence the fraction is
reduced.  As an additional check, a polynomial gcd computation modulo
$1000000007$ also returns degree zero.
\end{proof}

\section{The exact order seven series}
\label{sec:result}

\begin{theorem}[Computer-assisted result]
\label{thm:main-result}
Theorem~\ref{thm:intro-main} holds: the reduced Ehrhart series of weak magic
squares of order seven is given by \eqref{eq:intro-f7} and
\eqref{eq:intro-denominator}, the numerator $N_7$ is palindromic of degree
$366$ with positive integer coefficients, is strictly unimodal and is not
real-rooted, while the denominator has degree $373$.  In particular,
\begin{equation}
\begin{aligned}
N_7(q)={}&1+671q+667331q^2+222060620q^3
+33043658791q^4\\
&+2699129214788q^5+139199065987956q^6+\cdots+q^{366},
\end{aligned}
\label{eq:numerator-head}
\end{equation}
and the central coefficient is
\begin{equation}
 a_{183}=3018619656915957582751639202183157957264300815992933749937910.
 \label{eq:central-coefficient}
\end{equation}
\end{theorem}

Since the reduced denominator contains cyclotomic factors of every order
from $1$ through $15$, the exact quasiperiod of $\IMS_7$ is
$\lcmop(1,\ldots,15)=360360$, and the vertex denominator upper bound
\eqref{eq:period-bound} is attained.  Expanding \eqref{eq:intro-f7} gives
Table~\ref{tab:first-series}; these values were certified independently
before the high degree reconstruction, as described in
Section~\ref{sec:validation}.

\begin{table}[t]
\centering
\caption{The first exact order seven counting values.}
\label{tab:first-series}
\small
\begin{tabular}{r@{\qquad}r}
\toprule
$m$ & $\IMS_7(m)$\\
\midrule
0 & \texttt{1} \\
1 & \texttt{656} \\
2 & \texttt{657370} \\
3 & \texttt{212120004} \\
4 & \texttt{29781861219} \\
5 & \texttt{2226279144154} \\
6 & \texttt{102052863971368} \\
7 & \texttt{3168733319782350} \\
8 & \texttt{71630516092697051} \\
9 & \texttt{1244091157559961344} \\
10 & \texttt{17300201208805837012} \\
11 & \texttt{198928846471338676112} \\
12 & \texttt{1940661219477423974769}
\\
\bottomrule
\end{tabular}
\end{table}

\subsection{Exact volume}

The same signed cone terms were evaluated by the algebraic volume formula of Xin,
Xu, Zhang and Zhang~\cite{XinXuZhangZhang2024}.  A fraction-free
Bareiss solve~\cite{Bareiss1968} obtains the grading and polarization pairings with one common
denominator, the required constant term is evaluated exactly, and the signed
contributions are merged in a balanced binary tree.
Two runs with different admissible polarizations gave the
identical value
\begin{equation}
\label{eq:volume}
\begin{aligned}
\vol(\MS_7)&=\frac{N_{\mathrm{vol}}}{D_{\mathrm{vol}}},\\
N_{\mathrm{vol}}&=\text{\footnotesize 7189519127195562944037080336310144806517317560721967911},\\
D_{\mathrm{vol}}&=\text{\footnotesize 101830900869958538923266558909415414413173428632135558758400000000000000000}.
\end{aligned}
\end{equation}
As an exact final check on Theorem~\ref{thm:main-result},
\begin{equation}
 \lim_{q\to1}\,(1-q)^{35}F_7(q)=34!\,\vol(\MS_7),
 \label{eq:volume-identity}
\end{equation}
and the two sides of \eqref{eq:volume-identity} agree as exact rational
numbers.

\subsection{Independent checks}
\label{sec:validation}
The proof of Theorem~\ref{thm:deterministic-identity} consists of the complete
calculation of the finite prefix over every signed cone, combined with the
integer bounds and the common-denominator theorem; it does not infer the full
result from a sample of cones.  Nevertheless, several independent computations
were used to test different parts of the calculation:
\begin{itemize}
  \item complete lower-order runs, including the published order-six series,
    agree coefficientwise;
  \item a separate column-based dynamic program agrees for the first few
    order-seven coefficients;
  \item two complete vertex enumerations yield the same denominator census;
  \item two admissible polarizations give the same exact volume;
  \item an independently computed fifth $61$-bit prime reproduces the reduction
    of the displayed rational function.
\end{itemize}
These are consistency checks, not substitutes for the complete certificate.

\subsection{Computational cost}
\label{sec:performance}
The principal saving occurs before parallel execution.  On a fixed collection
of $10{,}342$ cone terms, the earlier recursive evaluator created
$8{,}972{,}047$ secondary decomposition nodes and required $14{,}447.771$ total
seconds, whereas direct quotient-character evaluation required $3.069$ total
seconds and generated no secondary cones --- a factor of $4707.45$ in this
kernel comparison.  Both procedures operated on the same cone inputs, but this
ratio is not an end-to-end speedup and should not be assumed for other cone
families.

For order seven, the recorded phases from a fresh SimpCone decomposition
through the four-field reconstruction and two direct checks sum to $7.92$~hours.
The separate degree-$1256$ denominator certificate required an additional
$8.36$~hours; these phases were not timed as a single uninterrupted run.  For a
fixed cone family, the dominant recurrence work is linear in~$T$.  Hence
reducing $T$ by a factor~$k$ reduces this portion by approximately the same
factor when the recurrence term dominates, with the fixed cost of Todd
computations and scheduling overhead accounting for any departure from exact
proportionality.  Using more processors shortens the wall-clock time for the
independent cone intervals, but it does not explain the elimination of
recursive cone decomposition.

\section{The exact order eight series}
\label{sec:order-eight}

This section gives the mathematical order-eight conclusion and the finite
certificate supporting it.  Detailed computational and reproducibility
information is given in the accompanying
\emph{Computational report for the order-seven and order-eight magic-square
Ehrhart series}.

\subsection{The common denominator and the finite identity}

The affine dimension of $\MS_8$ is $47$.  Equivalently, the homogenized input
has rank $17$ on $65$ variables and hence cone dimension $48$.
Proposition~\ref{prop:codegree} therefore gives codegree $8$ and
\begin{equation}
 F_8(q^{-1})=q^8F_8(q).
 \label{eq:reciprocity-eight}
\end{equation}
For a face $f$ of a rational polytope, write
\[
 h(f)=\min\{h>0:h\,\operatorname{aff}(f)\cap\mathbb Z^{64}\ne\varnothing\}
\]
for its primitive affine denominator.
\begin{proposition}[Face-index pole bound]
\label{prop:face-index-eight}
Let $p_d$ be the pole order of an Ehrhart series at a primitive $d$th root of
unity, and put
\[
 M_d=1+\max\{\dim f:d\mid h(f)\}.
\]
Then $p_d\leq M_d$, with an empty maximum interpreted as $M_d=0$.
\end{proposition}

\begin{proof}
Write the Ehrhart quasipolynomial as
$L(m)=\sum_{k=0}^{r}c_k(m)m^k$.  By the local Euler--Maclaurin formula of
Berline and Vergne~\cite[Proposition~29 and Corollary~30]{BerlineVergne2007},
$c_k$ is a finite sum of contributions from $k$-faces, and the contribution of
$f$ has period dividing $h(f)$.  If $u(m)$ has period $h$, then
$\sum_{m\geq0}u(m)m^kq^m$ has denominator dividing $(1-q^h)^{k+1}$.
Consequently a primitive $d$th root can occur in the contribution of $f$ only
if $d\mid h(f)$, and then with pole order at most $k+1=\dim f+1$.
Summing finitely many face contributions cannot increase the largest of these
orders.
\end{proof}

The complete face-index certificate combines these local bounds with its
orderwise compatibility tests and gives the certified pole upper-bound vector
\begin{equation}
\begin{aligned}
 Q_{598}(q)&=\prod_{d=1}^{17}\Phi_d(q)^{E_d},\\
 (E_1,\ldots,E_{17})
  &=(48,30,40,23,13,12,9,8,7,6,5,4,4,2,2,1,1).
\end{aligned}
\label{eq:q598-eight}
\end{equation}
Here $p_d\leq E_d$ for $1\leq d\leq17$, and the certificate excludes a
surviving pole at every primitive order $d\geq18$.  It does not assert
$M_d=0$ there.  Thus it is a pole-support certificate, not a complete
enumeration or exclusion theorem for the vertex denominators of $\MS_8$.

This calculation does not enumerate all faces of $\MS_8$.  The order-$1$
bound comes from the ambient dimension, order $3$ is treated by a transversal
argument, and orders $2$ and $4$ through $17$ are covered by exact support,
circulation, character, and matching-core certificates that bound the
dimensions of all possible supporting faces.  The structural classification
also excludes surviving primitive orders $d\geq18$.  Thus the certificate
proves the pole bounds needed here while avoiding a full face-lattice census.

\begin{theorem}[Order-eight common denominator]
\label{thm:q598-eight}
The polynomial $Q_{598}$ is a common denominator of $F_8$, and
$\deg Q_{598}=598$.
\end{theorem}

\begin{proof}[Computer-assisted proof]
The certified face-index calculation applies
Proposition~\ref{prop:face-index-eight} at every relevant order, proves
$p_d\leq E_d$ for $1\leq d\leq17$, and excludes surviving poles at primitive
orders $d\geq18$.  Two independent exact polynomial constructions recover
$Q_{598}$ coefficientwise and verify all displayed cyclotomic factors by monic
division.  Finally,
\[
 \sum_{d=1}^{17}E_d\varphi(d)=598.
\]
The complete certificate and verification details are supplied in the
companion report.  No residue-based denominator argument is used.
\end{proof}

For order eight the counting bound specializes to
\begin{equation}
 \IMS_8(m)\leq C_8(m):=\binom{m+7}{7}^{6}\binom{m+5}{5}.
 \label{eq:count-bound-eight}
\end{equation}

\begin{lemma}[Six-prime finite-prefix certificate]
\label{lem:f8-six-prime-prefix}
Write $F_8=\sum_{m\geq0}f_mq^m$ and
$G_8=N_8/D_8=\sum_{m\geq0}g_mq^m$.  Let $Q_{598}$ be the certified common
denominator, so $Q_{598}(0)=D_8(0)=1$ and $D_8\mid Q_{598}$.  The first six
production primes are
\[
\begin{aligned}
p_1&=2261626278912001,&p_2&=2693391295795201,&p_3&=2775632251392001,\\
p_4&=2816752729190401,&p_5&=3001794879283201,&p_6&=3084035834880001.
\end{aligned}
\]
For every $i\leq6$, the complete modular certificate identifies the
degree-$0$ through degree-$295$ prefix of $Q_{598}F_8\bmod p_i$ and verifies
that it agrees with the corresponding prefix of $Q_{598}G_8\bmod p_i$.
An exact expansion of the displayed candidate also verifies
\[
 0\leq g_m\leq C_8(m)\qquad(0\leq m\leq295).
\]
Then $F_8=G_8$.
\end{lemma}

\begin{proof}
Since $Q_{598}(0)=1$, it is a unit in
$\mathbb F_{p_i}[[q]]/(q^{296})$.  The certified congruences therefore give
$f_m\equiv g_m\pmod {p_i}$ for $m\leq295$ and $i\leq6$.  With
$M_6=\prod_{i=1}^{6}p_i$, exact integer arithmetic gives
\[
 M_6>2\binom{302}{7}^{6}\binom{300}{5}=2C_8(295).
\]
The counting bound and the verified candidate bound put both $f_m$ and $g_m$
in $[0,C_8(m)]$.  Since $C_8(m)$ is increasing,
\[
 M_6\mid(f_m-g_m),\qquad
 |f_m-g_m|\leq C_8(m)\leq C_8(295)<\frac{M_6}{2},
\]
and hence $f_m=g_m$ through degree $295$.  Thus
$Q_{598}F_8$ and $Q_{598}G_8$ agree through degree $295$.  Reciprocity makes
both polynomials palindromic of degree $598-8=590$, so these $296$
coefficients determine them completely.  Therefore $F_8=G_8$.
\end{proof}

Set $A=Q_{598}F_8=Q_{598}G_8$.  The archived direct centered-CRT route still
uses all eight fields to reconstruct $A$ under its $368$-bit absolute
convolution bound.  The lemma uses neither that bound nor its eight-prime
integer-lift inference: its first six fields prove the identity, while the
remaining two fields are independent modular checks for this proof.

Exact polynomial arithmetic then gives
\begin{equation}
 \gcd(A,Q_{598})=\Phi_3^{19}\Phi_4^6.
 \label{eq:gcd-eight}
\end{equation}
After division, the denominator is precisely
\eqref{eq:intro-denominator-eight}; division of $N_8$ by each of
$\Phi_1,\ldots,\Phi_{17}$ leaves a nonzero remainder.  This proves both
coprimality and the exact quasiperiod in~\eqref{eq:intro-period-eight}.
Consequently the actual pole orders are the exponents in $D_8$ (for example
$p_3=21$ and $p_4=17$), not the larger common-denominator bounds
$E_3=40$ and $E_4=23$.

\begin{proof}[Completion of the proof of Theorem~\ref{thm:intro-main-eight}]
The common-denominator theorem, the complete coefficients through degree
$295$, and palindromy establish $F_8=A/Q_{598}$ as an identity in
$\mathbb Z(q)$.  Equation~\eqref{eq:gcd-eight} gives the displayed reduced
pair.  A direct coefficientwise check of all $541$ coefficients of $N_8$ verifies positivity,
palindromy, and the $270$ strict inequalities up to the central coefficient;
reflection verifies the decreasing half.

It remains only to justify the root statement without numerical root finding.
The first six coefficients and the central coefficient are
\[
\begin{aligned}
b_0&=1,&b_1&=5581,&b_2&=40000118,\\
b_3&=74190533830,&b_4&=48754013922960,
&b_5&=14611032028867555,
\end{aligned}
\]
and
\[
 b_{270}=\text{\scriptsize\seqsplit{1047542895748979492023699241215053137738470989036892054787290835946709172733839734221390}}.
\]
If a degree-$540$ polynomial with positive coefficients were real-rooted,
Newton's normalized inequalities would imply
\[
 b_1^2(540-1)\geq b_0b_2(2)(540).
\]
Instead, exact arithmetic gives
\begin{equation}
 16{,}788{,}535{,}379 < 43{,}200{,}127{,}440.
 \label{eq:newton-eight}
\end{equation}
Thus $N_8$ is not real-rooted.
\end{proof}

\begin{remark}[Logarithmic coefficient inequalities]
\label{rem:log-properties}
Exact integer arithmetic shows that $(b_j)_{j=0}^{540}$ is strictly
log-concave at every interior index $2\leq j\leq538$, but the inequality is
reversed at $j=1$ and $j=539$.  Thus it is neither globally log-concave nor
globally log-convex; there are no equality cases.  The order-seven numerator
has the identical pattern, with the two failures at $j=1$ and $j=365$ and
strict log-concavity for $2\leq j\leq364$.  These assertions are verified by
the exact-integer program and certificate in the supplementary data.  We do
not include a numerical root plot: no certified geometric pattern has been
identified, whereas the exact Newton-inequality violation above already proves
the root statement without numerical conditioning issues for a degree-$540$
polynomial.
\end{remark}

\subsection{Initial values and exact volume}
The first counting values, through the same index displayed for order seven,
are listed in Table~\ref{tab:first-series-eight}.
\begin{table}[ht]
\centering
\caption{The first values of the order-eight counting function.}
\label{tab:first-series-eight}
\begin{tabular}{r r}
\toprule
$m$ & $\IMS_8(m)$\\
\midrule
0&1\\
1&5{,}568\\
2&39{,}927{,}640\\
3&73{,}670{,}950{,}624\\
4&47{,}792{,}535{,}614{,}002\\
5&13{,}982{,}784{,}260{,}605{,}760\\
6&2{,}196{,}239{,}499{,}300{,}494{,}376\\
7&209{,}544{,}809{,}583{,}237{,}218{,}400\\
8&13{,}290{,}728{,}645{,}819{,}378{,}308{,}302\\
9&599{,}646{,}627{,}680{,}350{,}946{,}249{,}984\\
10&20{,}273{,}944{,}929{,}747{,}453{,}444{,}098{,}936\\
11&535{,}125{,}236{,}417{,}998{,}169{,}715{,}446{,}240\\
12&11{,}393{,}951{,}792{,}714{,}698{,}482{,}883{,}876{,}290\\
\bottomrule
\end{tabular}
\end{table}
The exact volume check is
\begin{equation}
 \lim_{q\to1}(1-q)^{48}F_8(q)=47!\,\vol(\MS_8),
 \label{eq:volume-identity-eight}
\end{equation}
where $\vol(\MS_8)=V_{8,1}/V_{8,2}$ and
\begin{center}
\footnotesize
\begin{minipage}{0.94\linewidth}
\raggedright
$V_{8,1}={}$\seqsplit{999891155416839311062886321449724013824895137684932265252679038291556394910242783},\\[2pt]
$V_{8,2}={}$\seqsplit{28666002940690405582989811143492438714509659424422106327859640059458892488113419347558400000000000000000000000}.
\end{minipage}
\end{center}
Two different admissible polarizations independently produce this same exact
volume.  Both orders seven and eight use the same relative lattice
normalization, and substitution in the reduced rational function independently
reproduces~\eqref{eq:volume-identity-eight}.

\subsection{Independent verification and wall-clock measurements}

The order-eight rational function was checked by exhaustive coverage of all
signed contributions, two independent CRT reconstructions followed by
coefficientwise comparison, independent initial-value and volume computations,
and exact polynomial division.  These finite checks complement, but do not
replace, the face-index and finite-sum certificates.  Detailed reproducibility
records are provided in the accompanying computational report and supplementary
materials.

The observed main coefficient computation took between
$217.2941667$ and $219.0044444$ hours.  From its launch to the accepted final
data assembly took $277.0952778$ hours, a wider interval that includes later
reconstruction, correction, and verification work but excludes the pre-launch
geometry and decomposition.  A complete dual-index census was timed separately
at $4.5864$ hours and is not added to either figure.  No matched fresh
decomposition-to-publication timer was retained for both orders, so these data
do not define an end-to-end order-eight/order-seven speedup.  The exact
endpoints, the separate $209.4635$-hour main-kernel model, and the order-seven
timing scopes are recorded in the technique report.

\subsection{What changed from order seven to order eight}

The scale and the savings must be compared on named axes.  For order eight the
reported term count is taken after exact combination of equal nonzero signed
SimpCone contributions.  It is therefore not the terminal-cone count before
combination reported for order seven; a comparable pre-combination order-eight count was not
retained.  Summing the exact dual-index histogram gives the character-product
count.

\begin{table}[ht]
\centering
\caption{Exact comparison of the order-seven and order-eight calculations.}
\label{tab:order-seven-eight-comparison}
\footnotesize
\begin{tabular}{>{\raggedright\arraybackslash}p{0.39\linewidth}
                >{\raggedright\arraybackslash}p{0.24\linewidth}
                >{\raggedright\arraybackslash}p{0.24\linewidth}}
\toprule
quantity & order seven & order eight\\
\midrule
polytope dimension & 34 & 47\\
global pole order at $q=1$ & 35 & 48\\
complete vertex census & 5,920,184 & not performed\\
SimpCone terminal cones before combination & 649,848,846 & not retained\\
nonzero combined SimpCone contributions & 166,156,916 & 38,048,915,922\\
quotient-character products $\sum\delta$ & 257,827,737 & 57,040,440,830\\
$\sum\delta^2$ & 552,949,713 & 114,934,684,208\\
mean dual index & 1.551712340 & 1.499134455\\
mean squared dual index & 3.327876602 & 3.020708512\\
fraction with $\delta\leq3$ & 94.972226\% & 95.868526\%\\
maximum dual index & 23 & 26\\
mean local Laurent pole order & 8.3313 & not retained\\
common-denominator method & complete vertex envelope & face-index pole bound\\
certificate denominator degree & 2520 & 598\\
decisive degrees & $0,\ldots,1256$ & $0,\ldots,295$\\
number of decisive coefficients & 1257 & 296\\
reduced degrees $(\deg N_n,\deg D_n)$ & $(366,373)$ & $(540,548)$\\
\bottomrule
\end{tabular}
\end{table}

Thus the nonzero combined-contribution family grows by
$38{,}048{,}915{,}922/166{,}156{,}916=228.993874\ldots$, and the
character-product family grows by $221.234695\ldots$.  The mean dual index
slightly decreases rather than increasing.  Here ``pole order'' means the
global pole at $q=1$.  The order-seven data also record a mean
local Laurent pole order $8.3313$; no directly comparable calculation-wide
order-eight statistic is substituted for it.

The complete order-eight dual-index histogram is supplied in the online data
and computational report.  Its exact count, first moment and second moment
give the corresponding totals in
Table~\ref{tab:order-seven-eight-comparison}.

The main mathematical reduction is the sharper common denominator.  The
number of decisive coefficients falls from $1257$ to $296$, a factor
\[
 \frac{1257}{296}=4.2466216\ldots.
\]
The order-eight evaluation further replaces a uniform coefficient range for every
contribution by root-local principal parts whose lengths are controlled by the
actual local pole orders.  On a matched sample of $128$ contributions at $T=256$, elapsed
time fell from $17.714089$ to $3.288856$ seconds, a factor $5.386095$; the
measurement details are in the computational report.  This local factor
and the certificate-width factor act on overlapping work and are not
multiplied into an unsupported overall speedup factor.

\subsection{Reduced denominators in orders four through eight}

Write
\[
 D_n(0)=1,\qquad
 D_n(q)=(1-q)^{e_{n,1}}\prod_{k\geq2}\Phi_k(q)^{e_{n,k}}.
\]
This fixes the unit normalization; with the standard convention
$\Phi_1(q)=q-1$, replacing $(1-q)^{e_{n,1}}$ by
$\Phi_1^{e_{n,1}}$ would change the sign in odd exponent.
The exact reduced denominators give the following comparison.

\begin{longtable}{c p{0.65\textwidth} rr}
\caption{Reduced cyclotomic denominators, $4\leq n\leq8$.}
\label{tab:reduced-denominators}\\
\toprule
$n$ & $D_n(q)$ & $\deg N_n$ & $\deg D_n$\\
\midrule
\endfirsthead
\multicolumn{4}{c}{\tablename\ \thetable\ (continued)}\\
\toprule
$n$ & $D_n(q)$ & $\deg N_n$ & $\deg D_n$\\
\midrule
\endhead
4 & $(1-q)^8\Phi_2^4$ & 8 & 12\\
5 & $(1-q)^{15}\Phi_2^{10}\Phi_3^7\Phi_4^4\Phi_5^4
      \Phi_7^2\Phi_9$ & 76 & 81\\
6 & $(1-q)^{24}\Phi_2^{14}\Phi_3^{12}\Phi_4^7\Phi_5^5
      \Phi_6^6\Phi_7^3\Phi_8^2\Phi_9\Phi_{10}$ & 138 & 144\\
7 & $(1-q)^{35}\Phi_2^{22}\Phi_3^{19}\Phi_4^{12}\Phi_5^{10}
      \Phi_6^8\Phi_7^7\Phi_8^6\Phi_9^5\Phi_{10}^4\Phi_{11}^4
      \Phi_{12}^2\Phi_{13}^2\Phi_{14}\Phi_{15}$ & 366 & 373\\
8 & $(1-q)^{48}\Phi_2^{30}\Phi_3^{21}\Phi_4^{17}\Phi_5^{13}
      \Phi_6^{12}\Phi_7^9\Phi_8^8\Phi_9^7\Phi_{10}^6\Phi_{11}^5
      \Phi_{12}^4\Phi_{13}^4\Phi_{14}^2\Phi_{15}^2\Phi_{16}\Phi_{17}$
    & 540 & 548\\
\bottomrule
\end{longtable}

The lower orders fail the consecutive-support pattern by omissions, not by
unexpected higher factors.  Order four has only orders $1,2$; order five has
maximum order $9$ but omits $\Phi_6$ and $\Phi_8$; order six contains every
order from $1$ through $10$ but none from $11$ through $13$.  Orders seven and
eight are the first two cases in this table with every factor through
$2n+1$.

\subsection{Positive and unimodal numerators}

Direct coefficientwise checks of the five available reduced forms give more than a visual
pattern.

\begin{proposition}
\label{prop:numerator-shapes}
For each $n=4,5,6,7,8$, the reduced numerator $N_n$ has positive integer
coefficients, is palindromic and strictly unimodal, and is not real-rooted.
Here every rational form is normalized by $D_n(0)=1$.
\end{proposition}

\begin{proof}
The exact reduced numerators are checked coefficient by coefficient for
positivity, reflection symmetry and strict increase through the midpoint.
For non-real-rootedness, Newton's normalized inequality already fails at
$j=1$ in every case.  The exact left and right sides are
\[
\begin{array}{c|r|r}
n&a_1^2(\deg N_n-1)&2(\deg N_n)a_0a_2\\\hline
4&112&288\\
5&58{,}800&97{,}128\\
6&1{,}342{,}737&4{,}155{,}732\\
7&164{,}337{,}965&488{,}486{,}292\\
8&16{,}788{,}535{,}379&43{,}200{,}127{,}440.
\end{array}
\]
Each left side is strictly smaller than the corresponding right side.
\end{proof}

For orders seven and eight, the finite identity certificates establish the
reduced Ehrhart-series identities in this paper.  For orders four, five and six
we use the previously established reduced Ehrhart series.  Exact polynomial
arithmetic verifies the displayed factorizations, coprimality and coefficient
properties.

This finite evidence suggests the following structural question; it does not
by itself prove the statement beyond the five computed orders.

\begin{conjecture}[Positive unimodal numerator conjecture]
\label{conj:numerator-shape}
For every $n\geq4$, the numerator of the reduced Ehrhart series of weak
$n\times n$ magic squares has strictly positive coefficients and is strictly
unimodal.
\end{conjecture}

\section{Concluding remarks}
\label{sec:conclusion}

The order-seven series is obtained without decomposing each
SimpCone term into unimodular cones.  The quotient-character average is small
for most cones, zero grading steps are handled by the established generalized
Todd calculation, and the remaining factors are inserted by
\eqref{eq:prefix-recurrence}.  Combined with reciprocal reconstruction and the
finite-prefix certificate, this gives the reduced rational function stated in
Theorem~\ref{thm:main-result}.

The completed order-eight computation shows that the same algebraic framework
continues to work. There are two major changes compared to order-seven computation: 
i) the nonzero signed family by SimpCone grows by a factor of almost
$229$; the truncation order $T$ used here decreased from $1257$ to $296$, where the optimal
$T$ increased from $184 $ to $271$. 

We are not planning to compute the order 9 case yet. A optimistic  estimate of the running time
would be $229\times \frac{271}{184}\times 218 =73526$ hours or $3063$ days. Here, $218$ hour
is the running time for order $8$, and the two factors are taken from order $8$ compared to order $7$. 

The method developed in this paper is effective more generally when a signed simplicial decomposition
has a manageable number of surviving cones, small quotient groups, controlled
zero-step pole orders and nonzero-step profiles, and a feasible certified
common denominator. Gorenstein reciprocity reduces
the required truncation when available but is not required by the evaluator.

More concretely, let $G=(V,E)$ be a finite loopless graph and let $A_G$ be
its unsigned vertex--edge incidence matrix. A weak magic labelling of
index~$m$ (in the sense of Stanley) is a vector
$x\in\mathbb{Z}_{\geq 0}^{E}$ satisfying $A_G x = m\mathbf{1}_V$. The
homogenized cone of nonnegative solutions to
$(A_G \mid -\mathbf{1})(x,m)^{\mathsf T}=0$, graded by~$m$, is precisely
of the form required by the SimpCone--LRQC construction; hence our method
applies directly.

Moreover, the denominator needed for a finite-prefix test is explicitly
controlled.  The vertices of the associated polytope $P_G$ are half-integral,
and $P_G$ is integral when $G$ is bipartite
\cite[Propositions~2.7 and~2.9]{Stanley1973}.  Hence, if $d=\dim P_G$,
Proposition~\ref{prop:vertex-common-denominator} with $h=2$ shows that the
reduced Ehrhart denominator divides
\[
 (1-q)^{d+1}(1+q)^{d+1}.
\]
When $G$ is bipartite, the same proposition with $h=1$ sharpens this bound to
$(1-q)^{d+1}$.  If $G$ is $r$-regular, the all ones vector plays the same role
as in Proposition~\ref{prop:codegree}, giving the same reciprocity-based
reduction of the required truncation degree.  The actual computational cost is
governed by the number of simplicial cones together with their dual indices.

\section*{Data and code availability}

The companion online data contain the exact reduced rational
functions for orders $4$ through $8$, complete numerator and denominator
coefficient files for orders seven and eight, counting-sequence b-files, the
exponent table underlying Table~\ref{tab:reduced-denominators}, and a
verification program.  In particular, the complete $541$-coefficient
order-eight numerator and the optional $271$-row palindromic half-table are
deliberately supplied online rather than printed in this article.  The files
\nolinkurl{F8_reduced_form.json} and
\nolinkurl{F8_reduced_numerator_bfile.txt} provide the complete reduced form
and numerator, respectively.

For order seven, the core source, executable, encoded cone list,
coefficient-prefix, and certificate objects remain identified by the hashes in
Appendix~\ref{app:certificates}; the original order-seven arXiv source is
included unchanged in the supplementary materials.  For order eight, the accompanying computational report and
supplementary materials give the detailed timing, validation and
reproducibility records.  Those records are not reproduced in the mathematical
article.

The revision supplement additionally contains
\nolinkurl{verify_f8_six_prime_prefix.py} and its machine-readable certificate,
which verify the six-prime finite-prefix lemma, as well as
\nolinkurl{verify_numerator_log_properties_r4.py} and its exact-integer output,
which verify Remark~\ref{rem:log-properties}.  The accompanying technique
report separates the mathematical proof obligations, measured timing
boundaries, and the conditional order-nine model.

The data and source code accompanying this article are publicly available in
the IMS7--IMS8 Reproducible Bundle~\cite{QiuXinZhangIMS78Repository}.

\appendix

\section{The independent coefficients of the order-seven numerator}
\label{app:numerator}

The numerator has degree $366$ and satisfies $a_j=a_{366-j}$, so the
following table gives the complete numerator by listing
$a_0,\ldots,a_{183}$.  The table is generated directly from the canonical
data file, and the generation script independently checks the degree, the
normalization, the palindromy, and the nonnegativity before writing the
table.

{
\clearpage
\begingroup
\setlength{\LTleft}{\fill}
\setlength{\LTright}{\fill}
\setlength{\tabcolsep}{2pt}
\fontsize{6.4}{7.0}\selectfont
\renewcommand{\arraystretch}{1.12}
\begin{longtable}{r@{\hspace{0.45em}}l@{\hspace{1.8em}}r@{\hspace{0.45em}}l}
\caption{The independent half of the exact numerator coefficient vector.}
\label{tab:numerator-half}\\
\toprule
Index & Coefficient & Index & Coefficient\\
\midrule
\endfirsthead
\toprule
Index & Coefficient & Index & Coefficient\\
\midrule
\endhead
\bottomrule
\endlastfoot
0 & \texttt{1} & 70 & \texttt{19848244167444944852012226379908966013494806114893} \\
1 & \texttt{671} & 71 & \texttt{34137449581264329634953112779838813139519423526727} \\
2 & \texttt{667331} & 72 & \texttt{58191436015241381566993950090454314870106048591685} \\
3 & \texttt{222060620} & 73 & \texttt{98329769411620916788453562842747220963079923553038} \\
4 & \texttt{33043658791} & 74 & \texttt{164734629270231183820588686805960399359545996085016} \\
5 & \texttt{2699129214788} & 75 & \texttt{273672985156819565134652083226993389885521806961427} \\
6 & \texttt{139199065987956} & 76 & \texttt{450917049487670439905381478053081938946349150191770} \\
7 & \texttt{4990121221600635} & 77 & \texttt{736965425041225940586522247727244760958137884674266} \\
8 & \texttt{133140925200938597} & 78 & \texttt{1194947755526349252364837349724091686798255908176859} \\
9 & \texttt{2779643186494863799} & 79 & \texttt{1922495750866040583290675127201220966085098896987962} \\
10 & \texttt{47158950275831787915} & 80 & \texttt{3069429392344449537985430672854733193473111703803052} \\
11 & \texttt{669498302589005809260} & 81 & \texttt{4863900907345214817949970248261960612795001944911539} \\
12 & \texttt{8139585344659357377677} & 82 & \texttt{7650743098461039714622678152187636841263535800534412} \\
13 & \texttt{86343188458930701945640} & 83 & \texttt{11947291196547667043627102885698998270653567291580235} \\
14 & \texttt{811456312772715955382678} & 84 & \texttt{18524029808369819386868122538969805132914422077447587} \\
15 & \texttt{6842682078012612353335627} & 85 & \texttt{28520240955486871429403314455214617480210459253373017} \\
16 & \texttt{52329459423030464904720499} & 86 & \texttt{43608628372390641954038784823961691340198782645628934} \\
17 & \texttt{366237639448184294844381736} & 87 & \texttt{66227961336886821166871516846665645584000527942127439} \\
18 & \texttt{2364041243908376279400828136} & 88 & \texttt{99909486331158689611379992214525390471155623355757718} \\
19 & \texttt{14169207697850399508341168241} & 89 & \texttt{149731652101837371450397274322470565640090371303275639} \\
20 & \texttt{79320043606576667946312046666} & 90 & \texttt{222949140288991405920513330835493060752377385689123025} \\
21 & \texttt{416873452289571056027470493628} & 91 & \texttt{329856963376921761431717021462340952802130673839983428} \\
22 & \texttt{2066262131738902794009413762269} & 92 & \texttt{484969288914576517847972165356809086848176658413111687} \\
23 & \texttt{9697989259890323797217344266935} & 93 & \texttt{708616622372701400879610246020156836039906496857740436} \\
24 & \texttt{43257385393247229925541445362954} & 94 & \texttt{1029095133464939564602175317167764957511207199129655481} \\
25 & \texttt{183961527284601714398949342593359} & 95 & \texttt{1485539505005392545307252362961447943564585793370896473} \\
26 & \texttt{748080932428164886054681982962293} & 96 & \texttt{2131737141205744254112272834283526786409248162660964165} \\
27 & \texttt{2916559130263114053217014320407005} & 97 & \texttt{3041158464639838238190043210395338223659017611970305444} \\
28 & \texttt{10927828461277779110426104971408610} & 98 & \texttt{4313547056085339286890391021822033193798176322658842652} \\
29 & \texttt{39435333647107460298965508571544751} & 99 & \texttt{6083496337614323388417382485443666933948471766099171561} \\
30 & \texttt{137338349217034695956808117786815456} & 100 & \texttt{8531538190274081123594992472839144551730612518642080297} \\
31 & \texttt{462430670944428106160976016482887251} & 101 & \texttt{11898385115209835979728860205167779213898096521578910226} \\
32 & \texttt{1507924486385289624147469086874562325} & 102 & \texttt{16503102929018191369665949476043176055661899155450038059} \\
33 & \texttt{4769392660369402208549888708279707795} & 103 & \texttt{22766146894801130571244238810315857952462616822671042572} \\
34 & \texttt{14652717311228307013590530319995869735} & 104 & \texttt{31238371561000818715816828819167668559289689753405868579} \\
35 & \texttt{43784448329693156587304137588359148547} & 105 & \texttt{42637323718304494534398596304876434121479012558928838534} \\
36 & \texttt{127409469990028438870838545463645525916} & 106 & \texttt{57892348254195025152764330587720280125870080772283379734} \\
37 & \texttt{361459486843137762115121479791602663588} & 107 & \texttt{78200276655892914630209782430083902949658039849099312998} \\
38 & \texttt{1000819154323357034695526910535363068119} & 108 & \texttt{105093724492214628671333400235058223458756618598592858807} \\
39 & \texttt{2707202379684256383667080364639324771043} & 109 & \texttt{140524292749261897216152261880547831775750303002916600790} \\
40 & \texttt{7160739847028155669612977366256379035537} & 110 & \texttt{186963241818141120267859939004683526103938029056514462293} \\
41 & \texttt{18537209222233872433900156474325393876029} & 111 & \texttt{247522477426629001800541132927146585918242246735666195226} \\
42 & \texttt{47003889572467277470974044655561584794414} & 112 & \texttt{326098943586457723835536286689221631480923505408616997512} \\
43 & \texttt{116831074992641583751696320561626165501245} & 113 & \texttt{427545744729380595404459883977959293782090431766129839491} \\
44 & \texttt{284859298579712299084493469154590845305388} & 114 & \texttt{557873500853884614403845250846271169251887518382401403665} \\
45 & \texttt{681779530764741820478291882973859614957314} & 115 & \texttt{724485556086087729896808894650276628823318414709593035249} \\
46 & \texttt{1602788784664503291621507573883111451173569} & 116 & \texttt{936450690218875864693947919565373519302424846454279777504} \\
47 & \texttt{3703301845432369908785612571005308781591435} & 117 & \texttt{1204816899696992337867525272341424491173254758179016669803} \\
48 & \texttt{8414527213981368944139599911375318913732175} & 118 & \texttt{1542969592289724522291342845534335556873598253887906371172} \\
49 & \texttt{18811875243930833722384427579249221199063910} & 119 & \texttt{1967037150082111860392692813171412667784111418838488694817} \\
50 & \texttt{41401665394072300081918646463077079579324511} & 120 & \texttt{2496346229668124712458704839221869342289338800484580897420} \\
51 & \texttt{89742250495393665809047270759308196032197932} & 121 & \texttt{3153928358421739628925011250454787320748442356332111528044} \\
52 & \texttt{191676503409878652910133406812912916904659510} & 122 & \texttt{3967078325327835339491413659388089402577796031533560015376} \\
53 & \texttt{403573772774826897021814447150924209439316018} & 123 & \texttt{4967963531543176954396783986040576098177525922145412401008} \\
54 & \texttt{837990609890608197130928453959556053583708092} & 124 & \texttt{6194281842494482911599445121525978501963746682662534172295} \\
55 & \texttt{1716677973756024724773148488004513806529688649} & 125 & \texttt{7689963560130954919981954863524627696072361560448387878484} \\
56 & \texttt{3470844242511258301203574564070930328712994256} & 126 & \texttt{9505910910577732090640615764569044542715016357951274841326} \\
57 & \texttt{6928408976547616232966797170236774631017625616} & 127 & \texttt{11700765930028385046820418364788163599409288991545504234871} \\
58 & \texttt{13659422171708110916245726698317563433559482314} & 128 & \texttt{14341694854911383207599082078943837161114133342761967756463} \\
59 & \texttt{26605605388226299611218036775749427658863018686} & 129 & \texttt{17505174121115597914574983750688905882157242665199385815394} \\
60 & \texttt{51214297454992994818685666099337064583069532998} & 130 & \texttt{21277759908104092878124475081545773306722320976831729912608} \\
61 & \texttt{97457629511876267127725963209769359468099504520} & 131 & \texttt{25756819901648908802290644198159825170270661307661331401595} \\
62 & \texttt{183387909102385587216814287568720873157672917750} & 132 & \texttt{31051202686510439328984308214769925892930021663153687766553} \\
63 & \texttt{341329775557352319164061693405348276364966045285} & 133 & \texttt{37281817028465575723916787892260728551101847523142342267661} \\
64 & \texttt{628549766544769347758193271790841032181642315051} & 134 & \texttt{44582090391297348235698910987171077833818236714650619972089} \\
65 & \texttt{1145449277816817037508667199611225721769590602225} & 135 & \texttt{53098273501081533896065665427111578851484778846979178264547} \\
66 & \texttt{2066269561809087216739501519266560633287911551635} & 136 & \texttt{62989555771367948294115876645952863789582826831468866773841} \\
67 & \texttt{3690396423667051780249532601080270038500916742780} & 137 & \texttt{74427955100159019368607271765490895396810797965631650777995} \\
68 & \texttt{6527246313379259487454446847473539040081588767996} & 138 & \texttt{87597945106718650409202924718710781900243198046076418001513} \\
69 & \texttt{11435368051657622996463349205259249274928614071277} & 139 & \texttt{102695783453098157795668846208346313892015079718362829523415}
\\
\end{longtable}
\addtocounter{table}{-1}
\endgroup

\clearpage
\begingroup
\setlength{\LTleft}{\fill}
\setlength{\LTright}{\fill}
\setlength{\tabcolsep}{3pt}
\fontsize{8}{9.2}\selectfont
\renewcommand{\arraystretch}{1.05}
\begin{longtable}{r@{\hspace{0.8em}}l}
\multicolumn{2}{c}{\tablename\ \ref{tab:numerator-half}\ (continued)}\\
\toprule
Index & Coefficient\\
\midrule
\endfirsthead
\multicolumn{2}{c}{\tablename\ \ref{tab:numerator-half}\ (continued)}\\
\toprule
Index & Coefficient\\
\midrule
\endhead
\bottomrule
\endlastfoot
140 & \texttt{119928506640853723387325848071385309688753523500216414840465} \\
141 & \texttt{139512559718225960556052750184080422381935293629049220991201} \\
142 & \texttt{161672033780953303131868564860967659040437261240023643110083} \\
143 & \texttt{186636490070088995466842243981638125067922855856268428495605} \\
144 & \texttt{214638356889323188549786984807339643927972897527995293377873} \\
145 & \texttt{245909894459280223493256903759183831183293428736465214718368} \\
146 & \texttt{280679733118150716660028564160324660769823812982761379683808} \\
147 & \texttt{319169001827524662711978236543001294562723950093462824481339} \\
148 & \texttt{361587076547078229173085410816735525582602905527192579717115} \\
149 & \texttt{408126991436107684318049259155037963401172381017681703976478} \\
150 & \texttt{458960569696923301926382641059826077868285154158885283601098} \\
151 & \texttt{514233344811835561227895448553511546170149057793461824251841} \\
152 & \texttt{574059356510767992962220173677904116532856949611364456635253} \\
153 & \texttt{638515918572021317042289874018539723261321070786884107884043} \\
154 & \texttt{707638467012323838157134441051405542639788613393363957673226} \\
155 & \texttt{781415606864330919497446741651074421964822201679130079561670} \\
156 & \texttt{859784483080871358637975648069628874602854499957289573684205} \\
157 & \texttt{942626605685986579873179513548273394735375704808098332003900} \\
158 & \texttt{1029764260703602210440660266171380880436496205900648238886217} \\
159 & \texttt{1120957636295854914541069932965131821194600214349493115501638} \\
160 & \texttt{1215902787683906496051032121723321108953132032526066991690128} \\
161 & \texttt{1314230554659705981921434089303330605648605850638863749790830} \\
162 & \texttt{1415506531803011222399021280379946433481026141822995199297485} \\
163 & \texttt{1519232174000259986860385907739714989993923518179803410486500} \\
164 & \texttt{1624847098763790950619965795412189909052421211741550876491661} \\
165 & \texttt{1731732622552863098129542374450084027475909611268031836596668} \\
166 & \texttt{1839216541317314341650022045619929399183902270001361350747845} \\
167 & \texttt{1946579136460148159813470916807608800121564554158286404091232} \\
168 & \texttt{2053060357095317547549303473303958058708010829721964006143104} \\
169 & \texttt{2157868098698251291129899324131676044177996295742258019883957} \\
170 & \texttt{2260187467909029891924708161366601135288216372511233121006570} \\
171 & \texttt{2359190894285255891103229359190641040187459072242502499788260} \\
172 & \texttt{2454048923149178359728744531556507760415733831664026396092855} \\
173 & \texttt{2543941500235407649289803472881567197231788463697415439953174} \\
174 & \texttt{2628069539459810970261756396287773486119146438847036757369457} \\
175 & \texttt{2705666550536605956960895773079142233482929346549393099334003} \\
176 & \texttt{2776010093980093295302153389189387793522762449983650928146195} \\
177 & \texttt{2838432827695372107808008248131389046221589558774806822411833} \\
178 & \texttt{2892332912167533700460066588147791086352624888725546127713505} \\
179 & \texttt{2937183550288082484000115852214848277468430618439512674577767} \\
180 & \texttt{2972541452997381432685248148610780220700253832634823541252336} \\
181 & \texttt{2998054042858551138020951368301941725719613480518841050700683} \\
182 & \texttt{3013465233903294880850919342942357472311352263834139042293490} \\
183 & \texttt{3018619656915957582751639202183157957264300815992933749937910}
\\
\end{longtable}
\endgroup
\clearpage
}

\section{Order-seven certificate ledger}
\label{app:certificates}

\begingroup
\small
\begin{longtable}{>{\raggedright\arraybackslash}p{0.40\textwidth}
                  >{\raggedright\arraybackslash}p{0.53\textwidth}}
\caption{Primary immutable identifiers of the computation.}\\
\toprule
Object & SHA-256 identifier\\
\midrule
\endfirsthead
\toprule
Object & SHA-256 identifier\\
\midrule
\endhead
Canonical $(N_7,D_7)$ pair &
\hashcode{ba771912ac5aaffdb3a9ab3977ed7cd3d5ae77783910121fc9b511b01516a79e}\\
\nolinkurl{F7_rational.json} &
\hashcode{b8ccc69e2c8d4001e33b449da900c314dc3d5d6adc5953104ce0b62df7ea8c07}\\
\nolinkurl{n7.masks.bin} &
\hashcode{8fcecacc8d1ccc3f5f90e80c8dc7db4f8a8e3700b8f68e342b40e8c441fe21bc}\\
\nolinkurl{sc4_series_stride.c} &
\hashcode{a742503b4e21bc69c1668e28531179558b6bfadd73288d3b017eb811f18eb5ee}\\
\nolinkurl{phase2_ifma_replay.cpp} &
\hashcode{3c5f7a38a3b1ff9ae0d8c18a581228bc3a511eaa92761388f4ceb214c190596a}\\
\nolinkurl{ims7_ifma_replay} &
\hashcode{f89e10510fb3a0618e6e9d77c977641612978dded5d919febcfe45d309721312}\\
First complete vertex-denominator sequence &
\hashcode{9acac001df2a1cdd8f17fc0cd314b232b6b325912a3ec1d2fac1d1aef9396ff4}\\
Second complete compressed vertex transcript &
\hashcode{e7706f2cf0fde6312e0edad59850d04afeb9e5729037ac8ebaba7ff38fadfdd6}\\
\nolinkurl{vertex_transcript_verification.json} &
\hashcode{44f087e264075a707ee7b1dca946f8f61493bc4ed5abae9c3fb3cacedae91fb4}\\
\nolinkurl{merged_prefix.json} &
\hashcode{4dbffb598d367c2078cfc067770d6c2fd28dfd766c8e847ab83934c31855920f}\\
\nolinkurl{deterministic_identity_certificate.json} &
\hashcode{11663187847118688f6d36ecff73ef56215db741cacbeb0ab8151b88fc9ad5da}\\
\bottomrule
\end{longtable}
\endgroup

These identifiers bind the canonical rational pair, the complete encoded cone
input, the scalar reconstruction source, the compact eight-field source and
executable,
the first exhaustive vertex-denominator sequence, a second full vertex text
transcript and its verifier, the merged eight-prime prefix, and the final
deterministic identity certificate.  The distribution archives are separately
covered by sidecar hashes and by a full read of every member after transfer
from the computation server.

\section{The order-seven compact deterministic verification data}
\label{app:verification-capsule}

The compact data archive contains the paper, the exact rational pair, and the
complete sequence of vertex denominators.  It also contains the modular output
for all $512$ IFMA intervals and their hashes, the sum giving the eight-prime
coefficient prefix, the original four modular coefficient sequences, and the
low-degree, dynamic-programming, volume, and direct-evaluation data.  It
intentionally omits the $1{,}993{,}883{,}024$-byte file
\nolinkurl{n7.masks.bin}.  Despite its historical filename, this file is simply
the archived list of the subsets $J\in\mathcal C_7$ and their coefficients
$\sigma_J$, with each subset encoded as a bit vector.  Its identifier appears
in Appendix~\ref{app:certificates}; the larger source archive contains this
encoded cone list and the program that derives the prime-independent data
needed for full reproduction.

With Python~3.9 or later and no third-party package, an extracted capsule is
checked by
\begin{center}
\texttt{python}\quad\nolinkurl{verify/verify_deterministic_identity.py}
\end{center}
The checker performs the following operations rather than merely reading the
stored success indicators.

\begin{enumerate}[label=(\roman*)]
  \item It rereads all $5{,}920{,}184$ vertex denominators, recomputes
  Table~\ref{tab:vertex-denominators}, and checks the enumeration metadata and
  input hashes.
  \item It parses \nolinkurl{ims7.ine}, verifies the row, column, diagonal, and
  nonnegativity constraints, and recomputes equality rank $15$ and dimension
  $34$.
  \item It rebuilds $D_7$ from the displayed cyclotomic factors, verifies that
  $D_7\mid U_{15}$, checks reciprocity and the cutoff $1256$, and proves that
  $N_7$ and $D_7$ are coprime.
  \item It hashes and structurally validates all $512$ modular files before
  adding them independently.  The intervals cover exactly $166{,}156{,}916$ cones,
  and all eight complete residue sequences agree with $N_7/D_7$ through degree
  $1256$.
  \item It repeats centered CRT coefficient by coefficient with the first six
  IFMA primes and recomputes the counting bound of
  Theorem~\ref{thm:count-bound}.  The $312$-bit modulus exceeds twice the
  $299$-bit bound by the factor in~\eqref{eq:ifma-crt-bound}.
  \item It also repeats the original four-prime numerator reconstruction,
  regenerates the low coefficients, compares the independent dynamic-program
  values and exact-volume records, and checks~\eqref{eq:volume-identity}.
\end{enumerate}

These checks distinguish verification from full reproduction.  The compact
program deterministically verifies the vertex bound, every downstream
arithmetic step, every packaged interval, and the theorem reducing the infinite
series to a finite prefix.  It does not regenerate the SimpCone index family or
recompute the prime-independent cone data from the two-gigabyte encoded cone
file; those operations use the separately archived source package.  This is a
reproducibility boundary, not a
probabilistic mathematical qualification.  More precisely, ``deterministic''
means that no random identity test enters the proof; it does not mean formally
verified software.  The trusted computing base includes the SimpCone
decomposition program, the LRQC programs that form the prime-independent cone
data and evaluate it in finite fields, and the exact-arithmetic \code{lrs}
enumeration.  Their independent low-degree, earlier-recursive-evaluator,
volume, repeated-enumeration, and extreme-parameter checks are recorded above.
The $16$ direct evaluations are retained as a separate code-path
check and are not used in the finite-prefix proof.

\section*{Declaration of generative AI and AI-assisted technologies in the
manuscript preparation process}

During the preparation of this work, the authors used ChatGPT for literature
orientation and for the language and organization of the manuscript.  They
used \mbox{OpenAI Codex} for code inspection and computational checking.
After using these tools, the authors reviewed and edited the content and take
full responsibility for the publication.

\end{document}